\documentclass[12pt,a4paper]{article}

\usepackage[utf8]{inputenc}   
\usepackage{comment}          
\usepackage{xcolor}           
\usepackage{ulem}             
\usepackage{multirow}         
\usepackage{enumerate}        

\usepackage{amsmath}          
\usepackage{amssymb}          
\usepackage{amsfonts}         
\usepackage{mathtools}        
\usepackage{stmaryrd}         
\usepackage{mathrsfs}         
\usepackage{bbm}              

\usepackage{graphicx}         
\usepackage{subfigure}        
\usepackage{tikz-cd}          

\usepackage{pgfplots}         
\pgfplotsset{compat=1.18}     
\usetikzlibrary{arrows}       
\usetikzlibrary{calc}         
\usetikzlibrary{decorations.pathmorphing} 
\usetikzlibrary{shapes, backgrounds, fit, positioning, intersections} 

\usepackage{algorithm}
\usepackage{algorithmic}

\usepackage[colorlinks=true, linkcolor=blue, anchorcolor=blue, citecolor=blue]{hyperref} 

\makeatletter
\def\tank#1{\protected@xdef\@thanks{\@thanks \protect\footnotetext[0]{#1}}}
\def\bigfoot{\@footnotetext}
\makeatother

\newcommand{\ea}{\end{array}}

\usepackage{theorem} 
\newtheorem{theorem}{Theorem}[section]
\newtheorem{proposition}{Proposition}[section]
\newtheorem{corollary}{Corollary}[section]
\newtheorem{lemma}{Lemma}[section]
\newtheorem{definition}{Definition}[section]
\newtheorem{remark}{Remark}[section]
\newtheorem{example}{Example}[section]

{\theorembodyfont{\rmfamily}
}

\newenvironment{proof}{\noindent\textbf{Proof.}}{\hfill$\square$\par\medskip}

\begin{document}  

\title{\Large\bf Sharp Dimension Bounds for Spline Spaces over T-meshes with Highest Order of Smoothness}

\author{Bingru~Huang \quad and \quad Falai Chen\thanks{Corresponding author. E-mail address: chenfl@ustc.edu.cn} \\ 
School of Mathematical Sciences\\ 
University of Science and Technology of China\\ 
Hefei, 230026, People's Republic of China}
\date{}
\maketitle
\begin{center}
	\begin{minipage}{160mm}
{\bf Abstract.} The dimension of a polynomial spline space of bi-degree $(d_1,d_2)$ over a T-mesh $\mathscr{T}$ with the highest order of smoothness $(d_1-1,d_2-1)$ depends on both mesh topology and geometric configurations. Under the assumption that  the T-connected components of the T-mesh $\mathscr{T}$ contain no vanishable T $l$-edges, we develop explicit upper and lower bounds of the dimension of the polynomial spline space. By introducing a decoupling technique within the completely non-diagonalizable component (CNDC) of the T-mesh $\mathscr{T}$, we separate tightly coupled multi-vertex constraints and transform global conformality conditions into localized linear equations along each interior large edge. Based on the decoupling technique, a new dimension formula of the polynomial spline space is then presented, and from which sharp upper and lower bounds of the dimension are obtained. The bounds are sharp in the sense that different geometric realizations of T-meshes with the same topology can attain the lower and upper bounds for the dimension of the polynomial spline space. We further prove that the new formula is consistent with Mourrain's homological dimension formula, and a sharper lower bound is obtained by our method. 

\medskip

		\vfill
        
		\textbf{Keywords:} Polynomial spline space over T-mesh, dimension calculation, decoupling technique.  
	\end{minipage}
\end{center}
\section{Introduction}

Splines, which are piecewise polynomials with prescribed global smoothness, play an important role in approximation theory \cite{schumaker2007spline,lyche2018foundations}, computer-aided geometric design (CAGD) \cite{farin2002}, and numerical analysis. A major development in multivariate spline modeling was the introduction of Non-Uniform Rational B-Splines (NURBS) \cite{piegl1997}, which became a standard for geometric design and industrial manufacturing. However, NURBS have a structural limitation: they must be defined on tensor-product grids. Consequently, local refinement propagates across the entire mesh, introducing many redundant control points and reducing computational efficiency in adaptive simulations~\cite{ts1,ts2}.

To overcome this limitation and meet the increasing demand for local refinement—especially in Isogeometric Analysis (IGA) \cite{hughes2005, cottrell2009}, several types of locally refinable splines have been developed. Examples include hierarchical B-splines (HB-splines) \cite{HB1, HB2}, truncated hierarchical B-splines (THB-splines) \cite{thb}, locally refined splines (LR-splines) \cite{lr}, T-splines \cite{ts1, ts2}, PHT-splines \cite{PHT} and PT-splines \cite{zhong2025basis}. 
In fact, all these different types of splines can be regarded as a specific representation of {\bf polynomial splines over T-meshes}. To analyze such spine spaces, it is important to calculate the dimensions and construct bases of the spline spaces.

Over the past two decades, many methods have been proposed to calculate dimensions. Early studies used the B-net method \cite{dim2006} and homological algebra \cite{dim2014} to derive dimension formulas for splines with low degrees or low order of smoothness. Later, more general approaches were developed, including the minimal determining set method \cite{MDS}, the space embedding method \cite{jin2013}, and the smoothing cofactor method \cite{zeng2015, dim2016, huang2024stability}. The smoothing cofactor method converts the continuity conditions across mesh edges into linear algebraic equations, reducing the dimension problem into finding the rank of a global conformality matrix~\cite{zeng2015}.

However, when the degree $d$ of the spline and the the smoothness order $\mu$ are close, especially for the highest order of smoothness where $\mu = d-1$, the dimension of the spline space can be unstable \cite{Ins2011, Ins2012}. In this case, the dimension depends not only on the topology of the mesh (such as the number of vertices, edges, and cells), but also on the exact geometric coordinates of the grid lines \cite{guo2015problem, li2019instability}. This dimensional instability complicates the use of these splines in practical engineering applications \cite{zhong2025basis}.

To address this problem, many researchers have focused on special classes of T-meshes with stable dimensions, such as hierarchical T-meshes \cite{zeng2015,huang2026dimension} and diagonalizable T-meshes \cite{dim2016, huang2024stability}. Huang and Chen \cite{huang2024stability} developed a decomposition theory that separates the T-connected component into a diagonalizable component and a completely non-diagonalizable component (CNDC). Under this decomposition, dimensional stability is reduced to the coordinate-independence of the rank of the conformality matrix attached to the CNDC \cite{huang2024stability, huang2025preliminarystudydimensionalstability}.

Nevertheless, for meshes with complex  topologies, analyzing the conformality matrix of a CNDC is still difficult. To address this issue, in this paper we introduce the concept of \textit{decoupled multi-vertex cofactor} in order to separate the cofactors at the multi-vertices. Under this technique, the global conformality conditions are separated into two complementary levels. The first level consists of local equations along individual horizontal or vertical T $l$-edges, while the second level only requires the two copies assigned to the same multi-vertex to be compatible.  The local equations can therefore be solved independently edge by edge, and the remaining compatibility requirements are assembled only after these local solutions have been obtained. This gives a localized route to dimension calculation and dimension bounds.

The main contributions of this paper are summarized as follows:
\begin{itemize}
    \item[\textbf{(1)}] We introduce a decoupling technique within the CNDC to effectively separate tightly coupled multi-vertex constraints.
    \item[\textbf{(2)}] We transform the global conformality conditions into localized linear equations along individual T $l$-edges, providing a direct and general method for spline dimension computation.
    \item[\textbf{(3)}] We derive explicit upper and lower dimension bounds for spline spaces over T-meshes with no vanishable T $l$-edges.  The bounds are sharp in the sense that they are attainable for specific geometric realizations of a fixed T-mesh topology.
   \item[\textbf{(4)}] The relationship between our dimension formula and the Mourrain's homological dimension formula is presented. 
\end{itemize}

The remainder of this paper is organized as follows. Section~2 reviews polynomial spline spaces over T-meshes, the smoothing cofactor method, and the complete partition of a T-mesh. Section~3 develops the decoupling technique and the localized dimension formula. Section~4 establishes the upper and lower bounds of the dimension, presents their precise relationship with the Mourrain's homological formula. Section~5 concludes the paper and outlines future work.

\section{Preliminaries}\label{sec:prelim}
In this section, we review some  preliminary knowledge regarding T-meshes, polynomial spline spaces, and the smoothing cofactor method, including conformality vector spaces, conformality matrices, and the decomposition of T-connected components of a T-mesh.

\subsection{Spline spaces over T-meshes}\label{subsec:spline_spaces}
We first introduce some necessary terminologies and notations. A T-mesh $\mathscr{T}$ is an axis-aligned rectangular partition of a simply-connected domain that allows T-junctions. Vertices are the grid points of the mesh; among them, T-junctions are interior vertices of valence 3. An edge is a line segment connecting two adjacent vertices along a horizontal or vertical grid line. Edges are classified as boundary edges or interior edges according to their geographic location.

A large edge (denoted as an $l$-edge) is a maximal straight segment composed of collinear contiguous edges, whose endpoints are either T-junctions or boundary vertices. An interior $l$-edge is further categorized as follows:
\begin{itemize}
    \item A \textit{cross-cut} if both endpoints lie on the domain boundary.
    \item A \textit{T $l$-edge} if both endpoints are T-junctions.
    \item A \textit{ray} otherwise (one endpoint is a T-junction and the other lies on the boundary).
\end{itemize}

The set of all T $l$-edges in a T-mesh $\mathscr{T}$ forms a \textbf{T-connected component}, denoted as $T(\mathscr{T})$, which plays a central role in the dimension formula. In this paper, we assume without loss of generality that the T-connected component is connected. Otherwise, a T-connected component can be divided into several isolated connected components, each of which can be treated independently. For convenience, we classify the interior vertices of a T $l$-edge into two classes: \textbf{mono-vertices} and \textbf{multi-vertices}. If an interior vertex is the intersection of two T $l$-edges, then it is called a multi-vertex. Otherwise, it is called a mono-vertex.

Figure~\ref{fig:T-connected component} provides an illustrative example of a T-mesh and its corresponding T-connected component. Within the T-connected component, there are 4 T $l$-edges: $v_1v_{7}, v_2v_8, v_3v_4,$ and $v_5v_{6}$.

\begin{figure}[htbp]
\centering
\begin{minipage}[b]{0.48\textwidth}
\centering
\begin{tikzpicture}[line cap=round,line join=round,>=triangle 45,x=1.0cm,y=1.0cm,scale=1]
\draw [line width=1.pt] (1.,1.)-- (6.,1.);
\draw [line width=1.pt] (1.,1.)-- (1.,6.);
\draw [line width=1.pt] (1.,6.)-- (6.,6.);
\draw [line width=1.pt] (6.,6.)-- (6.,1.);
\draw [line width=1.pt] (1.,5.)-- (6.,5.);
\draw [line width=1.pt] (1.,2.)-- (6.,2.);
\draw [line width=1.pt] (2.,6.)-- (2.,1.);
\draw [line width=1.pt] (1.5,6.)-- (1.5,1.);
\draw [line width=1.pt] (1.,5.5)-- (6.,5.5);
\draw [line width=1.pt] (1.,1.5)-- (6.,1.5);
\draw [line width=1.pt] (5.5,6.)-- (5.5,1.);
\draw [line width=1.pt] (5.,1.)-- (5.,5.5);

\draw [line width=2.pt] (3,5.5)-- (3.,2.);
\draw [line width=2.pt] (1.5,4)-- (5.5,4.);
\draw [line width=2.pt] (2.,3.)-- (5.5,3);
\draw [line width=2.pt] (4,1.5)-- (4.,5.);

\begin{scriptsize}
\draw [fill=black] (3,5.5) circle (2.0pt);
\draw[color=black] (3.13,5.65) node {$v_1$};
\draw [fill=black] (3.,2.) circle (2.0pt);
\draw[color=black] (3.14,2.1) node {$v_7$};
\draw [fill=black] (1.5,4) circle (2.0pt);
\draw[color=black] (1.64,4.15) node {$v_3$};
\draw [fill=black] (5.5,4.) circle (2.0pt);
\draw[color=black] (5.64,4.15) node {$v_4$};
\draw [fill=black] (2.,3.) circle (2.0pt);
\draw[color=black] (2.13,3.15) node {$v_5$};
\draw [fill=black] (5.5,3) circle (2.0pt);
\draw[color=black] (5.65,3.15) node {$v_6$};
\draw [fill=black] (4,1.5) circle (2.0pt);
\draw[color=black] (4.13,1.34) node {$v_8$};
\draw [fill=black] (4.,5.) circle (2.0pt);
\draw[color=black] (4.13,5.15) node {$v_2$};
\end{scriptsize}
\end{tikzpicture}
\vfill
\centering{(a)A T-mesh $\mathscr{T}$}
\end{minipage}
\hfill 
\begin{minipage}[b]{0.48\textwidth}
\centering
\begin{tikzpicture}[line cap=round,line join=round,>=triangle 45,x=1.0cm,y=1.0cm,scale=1]

\draw [line width=1.pt] (3,5.5)-- (3.,2.);
\draw [line width=1.pt] (1.5,4)-- (5.5,4.);
\draw [line width=1.pt] (2.,3.)-- (5.5,3);
\draw [line width=1.pt] (4,1.5)-- (4.,5.);

\begin{scriptsize}
\draw [fill=black] (3,5.) circle (2.0pt);
\draw [fill=black] (3,3.) circle (2.0pt);
\draw [fill=black] (3,4.) circle (2.0pt);
\draw [fill=black] (4,3.) circle (2.0pt);
\draw [fill=black] (4,4.) circle (2.0pt);
\draw [fill=black] (5,4.) circle (2.0pt);
\draw [fill=black] (5,3.) circle (2.0pt);
\draw [fill=black] (2,4.) circle (2.0pt);
\draw [fill=black] (3,5.5) circle (2.0pt);
\draw[color=black] (3.13,5.65) node {$v_1$};
\draw [fill=black] (3.,2.) circle (2.0pt);
\draw[color=black] (3.14,2.1) node {$v_7$};
\draw [fill=black] (1.5,4) circle (2.0pt);
\draw[color=black] (1.64,4.15) node {$v_3$};
\draw [fill=black] (5.5,4.) circle (2.0pt);
\draw[color=black] (5.64,4.15) node {$v_4$};
\draw [fill=black] (2.,3.) circle (2.0pt);
\draw[color=black] (2.13,3.15) node {$v_5$};
\draw [fill=black] (5.5,3) circle (2.0pt);
\draw[color=black] (5.65,3.15) node {$v_6$};
\draw [fill=black] (4,1.5) circle (2.0pt);
\draw[color=black] (4.13,1.34) node {$v_8$};
\draw [fill=black] (4.,5.) circle (2.0pt);
\draw[color=black] (4.13,5.15) node {$v_2$};
\end{scriptsize}
\end{tikzpicture}
\vfill
\centering{(b) T-connected component $T(\mathscr{T})$}
\end{minipage}

\caption{\label{fig:T-connected component}A T-mesh and the corresponding T-connected component}
\end{figure}

Let $\mathscr{F}$ denote the set of all rectangular faces. The polynomial spline space of bi-degree $(d_1,d_2)$ with the smoothness order $(\mu_1,\mu_2)$ over $\mathscr{T}$ is defined as
$$S_{d_1,d_2}^{\mu_1,\mu_2}(\mathscr{T}) = \left\{ s(x,y) \in C^{\mu_1,\mu_2}(\Phi) : s|_{\phi} \in \mathbb{P}_{d_1} \otimes \mathbb{P}_{d_2},\ \forall \phi \in \mathscr{F} \right\},$$
where $\mathbb{P}_{d_1} \otimes \mathbb{P}_{d_2}$ is the tensor-product polynomial space of bidegree $(d_1,d_2)$, and $C^{\mu_1,\mu_2}(\Phi)$ consists of bivariate functions that are $\mu_1$-times continuously differentiable in the $x$-direction and $\mu_2$-times in the $y$-direction over the domain $\Phi$ occupied by $\mathscr{F}$. This paper focuses on the highest order of smoothness case: $d_1 = d_2 = d$ and $\mu_1 = \mu_2 = d-1$. The spline space is simply denoted as $S_d(\mathscr T)$ in this case.

According to the smoothing cofactor method \cite{zeng2015}, the $C^{d-1}$ continuity constraints across a horizontal T $l$-edge with vertices of $x$-coordinates $x_1 < x_2 < \cdots < x_r$ translate into linear relations among the associated vertex cofactors $\boldsymbol{\delta}_l = (\delta_1, \delta_2, \dots, \delta_r)^T \in \mathbb{R}^r$ (see Figure~\ref{fig: vertex cofactors}):
\begin{equation}\label{eq4}
\sum\limits_{i=1}^{r}\delta_{i}(x-x_i)^d=0.
\end{equation}

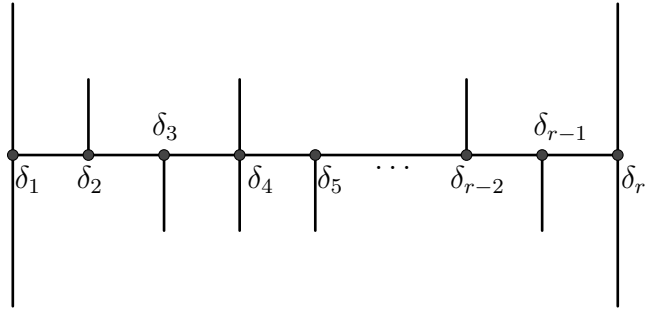
\begin{figure}[htbp]
    \centering
\definecolor{black}{rgb}{0.26666666666666666,0.26666666666666666,0.26666666666666666}
\begin{tikzpicture}[line cap=round,line join=round,>=triangle 45,x=1cm,y=1cm,scale=2]
\draw [line width=1pt] (1,0.5)-- (1,2.5);
\draw [line width=1pt] (1,1.5)-- (5,1.5);
\draw [line width=1pt] (5,2.5)-- (5,0.5);
\draw [line width=1pt] (1.5,1.5)-- (1.5,2);
\draw [line width=1pt] (2,1.5)-- (2,1);
\draw [line width=1pt] (2.5,2)-- (2.5,1);
\draw [line width=1pt] (3,1.5)-- (3,1);
\draw [line width=1pt] (4.5,1.5)-- (4.5,1);
\draw [line width=1pt] (4,1.5)-- (4,2);
\draw (0.94,1.5) node[anchor=north west] {$\delta_1$};
\draw (1.35,1.5) node[anchor=north west] {$\delta_2$};
\draw (1.85,1.85) node[anchor=north west] {$\delta_3$};
\draw (2.48,1.5) node[anchor=north west] {$\delta_4$};
\draw (2.94,1.5) node[anchor=north west] {$\delta_5$};
\draw (3.82,1.5) node[anchor=north west] {$\delta_{r-2}$};
\draw (4.37,1.85) node[anchor=north west] {$\delta_{r-1}$};
\draw (4.95,1.5) node[anchor=north west] {$\delta_r$};
\draw (3.32,1.5) node[anchor=north west] {$\ldots$};
\begin{scriptsize}
\draw [fill=black] (1,1.5) circle (1pt);
\draw [fill=black] (1.5,1.5) circle (1pt);
\draw [fill=black] (2,1.5) circle (1pt);
\draw [fill=black] (2.5,1.5) circle (1pt);
\draw [fill=black] (3,1.5) circle (1pt);
\draw [fill=black] (4,1.5) circle (1pt);
\draw [fill=black] (4.5,1.5) circle (1pt);
\draw [fill=black] (5,1.5) circle (1pt);
\end{scriptsize}
\end{tikzpicture}
    \caption{\label{fig: vertex cofactors}Vertex cofactors along a horizontal T $l$-edge}
\end{figure}

Expanding Eq.~\eqref{eq4} yields a local homogeneous linear system $\mathbf{P}_l \boldsymbol{\delta}_l = \mathbf{0}$, expressed in matrix form as:
\begin{equation}
\begin{pmatrix}
	1 & 1 & \cdots & \cdots & 1 \\
	x_1 & x_2 & \cdots & \cdots & x_r\\
	x_1^2 & x_2^2 & \cdots & \cdots & x_r^2\\
	\cdots & \cdots & \cdots & \cdots & \cdots\\
	x_1^{d-1} & x_2^{d-1} & \cdots & \cdots & x_r^{d-1}\\
	x_1^d & x_2^d & \cdots & \cdots & x_r^d
\end{pmatrix}
\begin{pmatrix}
	\delta_1 \\
	\delta_2 \\
	\delta_3 \\
	\vdots\\
	\delta_{r-1} \\
	\delta_{r} 
\end{pmatrix}
=\begin{pmatrix}
	0 \\
	0 \\
	0 \\
	\vdots\\
	0 \\
	0
\end{pmatrix}.
\label{gcc}
\end{equation}

For a T $l$-edge $l$ containing $n(l)$ vertex cofactors, the coefficient matrix in \eqref{gcc} is a $(d+1)\times n(l)$ Vandermonde matrix. When $n(l)\le d+1$, the Vandermonde matrix has full column rank, and its null space is zero. Therefore all the vertex cofactors on $l$ must vanish, that is,  $l$ contributes no  degree of freedom to the dimension of the spline space. In this case, we call $l$ \emph{vanishable}. Throughout this paper, we assume that the T-connected components do not contain any vanishable T $l$-edge, or equivalently,
\[
 n(l)\ge d+2 \qquad\text{for every T $l$-edge }l.
\]
Consequently every Vandermonde matrix associated with a T l-edge has row rank $d+1$ and local nullity $n(l)-d-1\ge1$ ~\cite{huang2025preliminarystudydimensionalstability}. 

By expanding the local coefficient matrix $\mathbf{P}_{l_i}$ to match the global cofactor vector $\boldsymbol{\delta} = (\delta_1, \dots, \delta_v)^T \in \mathbb{R}^v$ via zero-padding for vertices not belonging to $l_i$, we obtain a globalized system of linear equations, denoted as $\mathscr{P}_{l_i} \boldsymbol{\delta} = \mathbf{0}$.

\begin{definition}[\cite{zeng2015}]\label{def:cvs}
Let $T(\mathscr{T})$ be the T-connected component of $\mathscr{T}$ consisting of T $l$-edges $l_1, \dots, l_t$. The conformality vector space (\textsf{CVS}) is defined as:
$$\textsf{CVS}[T(\mathscr{T})] := \bigl\{ \boldsymbol{\delta} \in \mathbb{R}^v : \mathscr{P}_{l_i} \boldsymbol{\delta} = \mathbf{0},\ 1 \leq i \leq t \bigr\}.$$
The coefficient matrix associated with the global homogeneous system is called the global \textbf{conformality matrix}, denoted by $M(T(\mathscr{T}))$.
\end{definition}

The dimension formula is then given by the following result~\cite{zeng2015}.
\begin{theorem}[\cite{zeng2015}]\label{thm: dim formula}
For a T-mesh $\mathscr{T}$,
$$\dim S_d(\mathscr{T}) = (d+1)^2 + c(d+1) + n_v - \mathrm{rank}\bigl(M(T(\mathscr{T}))\bigr),$$
where $c$ is the number of cross-cuts and $n_v$ is the number of interior vertices.
\end{theorem}
Thus, the study of $\dim S_d(\mathscr T)$ pins down to the study of $\mathrm{rank}(M(T(\mathscr{T})))$. 


\subsection{Decomposition of T-connected components}\label{subsec:decomposition}

By Theorem~\ref{thm: dim formula}, computing the spline space dimension reduces to finding the rank of the global conformality matrix associated with the T-connected component $T(\mathscr{T})$. To analyze where the dimensional instability occurs, this subsection introduces a decomposition framework for $T(\mathscr{T})$ based on \cite{huang2024stability}. We first introduce diagonalizable T-meshes, which have stable dimensions and serve as the basis for partitioning general T-connected components into diagonalizable part and non-diagonalizable part.

\begin{definition}[\cite{dim2016}]\label{def:reasonable_order}
For a given T-mesh $\mathscr{T}$, if there exists an ordering of all the T $l$-edges of its T-connected component $T(\mathscr{T})$, say $l_1 \succ l_2 \succ \ldots \succ l_t$, such that $r(l_i) \geq d+1$, where $r(l_i)$ is the number of remaining vertices on $l_i$ after removing the intersection vertices of $l_i$ with the preceding T $l$-edges $l_j$ ($j=1,2,\cdots,i-1$), then this ordering is called a \textbf{reasonable order} and $\mathscr{T}$ is called a \textbf{diagonalizable T-mesh}.
\end{definition}

According to \cite{dim2016}, the dimension of the spline space $S_d(\mathscr{T})$ over a diagonalizable T-mesh is stable and can be explicitly computed as follows.

\begin{proposition}[\cite{dim2016}]\label{prop:diagonalizable_dim}
Suppose that a T-mesh $\mathscr{T}$ is diagonalizable. Then the dimension of the spline space $S_d(\mathscr{T})$ is stable and can be explicitly computed as:
\begin{equation}\label{eq:diagonalizable_formula}
	\dim S_{d}(\mathscr{T})=(d+1)^2+(c-t)(d+1)+n_v,
\end{equation}
where $c$ is the number of cross-cuts of $\mathscr{T}$, $t$ is the number of T $l$-edges of $T(\mathscr{T})$, and $n_v$ is the number of interior vertices.
\end{proposition}

\begin{example}\label{ex:diagonalizable_tmesh}
Consider the polynomial spline space $S_3(\mathscr{T})$ over the T-mesh $\mathscr{T}$ illustrated in Fig.~\ref{fig:T-connected component}(a). The associated T-connected component $T(\mathscr{T})$ consists of four T $l$-edges: $v_1v_7$, $v_5v_6$, $v_2v_8$, and $v_3v_4$, as depicted in Fig.~\ref{fig:T-connected component}(b). 

We can verify that $\mathscr{T}$ is a diagonalizable T-mesh by constructing a reasonable ordering for these T-edges. Specifically, let the sequence of the interior edges be sorted as follows:
\begin{equation*}
v_1v_7 \succ v_5v_6 \succ v_2v_8 \succ v_3v_4.
\end{equation*}
According to this specific ordering, the number of vertices $r(\cdot)$ on each edge can be sequentially evaluated as:
\begin{equation*}
r(v_1v_7) = 5, \quad r(v_5v_6) = 4, \quad r(v_2v_8) = 4, \quad \text{and} \quad r(v_3v_4) = 4.
\end{equation*}
They all satisfy the condition $r(\cdot)\ge d+1$, thus this ordering is a reasonable ordering. The existence of such a reasonable ordering guarantees that the T-mesh $\mathscr{T}$ is a diagonalizable T-mesh.
\end{example}

In general, a T-connected component may not be diagonalizable. In this case, we decompose the component to separate the non-diagonalizable parts from the diagonalizable ones. We first define a sub-component and a bipartite partition of $T(\mathscr{T})$.

\begin{definition}[\cite{huang2024stability}]\label{def:bipartite_partition}
A subset of T $l$-edges of $T(\mathscr{T})$, together with all the vertices lying on them, is called a sub-component of $T(\mathscr{T})$.
Let $T(\mathscr{T})$ be the T-connected component of a T-mesh $\mathscr{T}$ and $T \subset T(\mathscr{T})$ be a sub-component. Then we call $\{T, T(\mathscr{T})\setminus T\}$ a \textbf{bipartite partition} of $T(\mathscr{T})$. 
\end{definition}

Based on the diagonalizability of sub-components, the regular and complete partitions are defined as follows.

\begin{definition}[\cite{huang2024stability}]\label{def:regular_partition}
Let $T(\mathscr{T})$ be the T-connected component of a non-diagonalizable T-mesh $\mathscr{T}$, and let $\{T_1, T_2\}$ be a bipartite partition of $T(\mathscr{T})$. If the sub-component $T_1$ does not admit a reasonable order while the sub-component $T_2$ does, then $T_1$ is called a \textbf{non-diagonalizable component} and $T_2$ is called a \textbf{diagonalizable component} of $T(\mathscr T)$. Under these conditions, $\{T_1, T_2\}$ is referred to as a \textbf{regular partition} of $T(\mathscr T)$.
\end{definition}

\begin{definition}[\cite{huang2024stability}]\label{def:complete_partition}
Let $\mathscr{T}$ be a non-diagonalizable T-mesh and $T(\mathscr{T})$ be its T-connected component. Suppose $\{T_1, T_2\}$ is a regular partition of $T(\mathscr{T})$ with $T_1 \subset T(\mathscr{T})$. If there does not exist another regular partition $\{T_1', T_2'\}$ of $T(\mathscr{T})$ such that $T_1'$ is a proper subset of $T_1$, then $\{T_1, T_2\}$ is called a \textbf{complete partition} of $T(\mathscr{T})$, and $T_1$ is called the \textbf{completely non-diagonalizable component} (CNDC) of $T(\mathscr{T})$.
\end{definition}

An algorithm to compute the complete partition of $T(\mathscr{T})$ was proposed in \cite{huang2024stability}. Although the intermediate verification steps in the practical implementation depend on the chosen ordering of the T $l$-edges, it was proved in \cite{huang2025preliminarystudydimensionalstability} that the final decomposition result is unique and independent of the edge ordering.

\begin{proposition}[\cite{huang2025preliminarystudydimensionalstability}]\label{prop:uniqueness_complete_partition}
The complete partition of a $T$-connected component $T(\mathscr{T})$ is unique, implying that the number of $T$ $l$-edges and vertices in CNDC constitutes a topological invariant.
\end{proposition}

\begin{example}\label{ex:complete_partition}
Consider the polynomial spline space $S_3(\mathscr{T})$ over a T-mesh $\mathscr{T}$ shown in Fig.~\ref{fig: a general T-mesh}. The associated T-connected component $T(\mathscr{T})$ consists of five T $l$-edges: $v_1v_7$, $v_3v_4$, $v_5v_6$, $v_2v_8$, and $v_9v_{10}$, which are collectively detailed in Fig.~\ref{fig:T-mesh decomposition framework}(a).

It is easy to verify that these edges cannot form a reasonable ordering under any possible permutation; thus, $\mathscr{T}$ is not a diagonalizable T-mesh. Therefore, we can perform a complete partition on $T(\mathscr{T})$ according to the algorithm proposed in \cite{huang2024stability}. The detailed partition results are illustrated in Fig.~\ref{fig:T-mesh decomposition framework}(b) and (c).
\end{example}

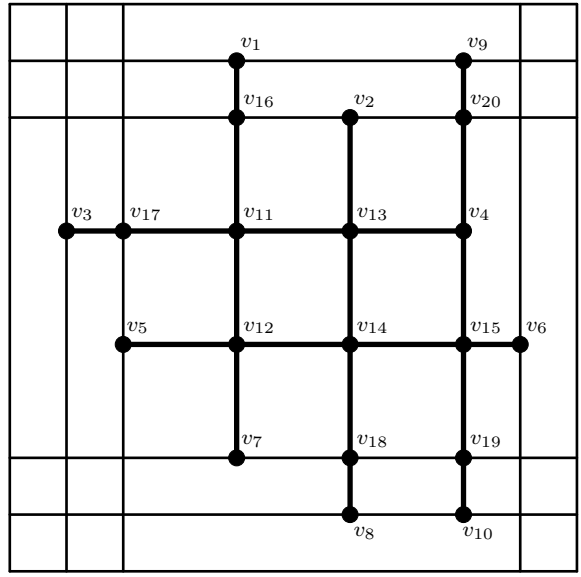
\begin{figure}
\centering
\begin{tikzpicture}[line cap=round,line join=round,>=triangle 45,x=1.0cm,y=1.0cm,scale=1.5]
\draw [line width=1.pt] (1.,1.)-- (6.,1.);
\draw [line width=1.pt] (1.,1.)-- (1.,6.);
\draw [line width=1.pt] (1.,6.)-- (6.,6.);
\draw [line width=1.pt] (6.,6.)-- (6.,1.);
\draw [line width=1.pt] (1.,5.)-- (6.,5.);
\draw [line width=1.pt] (1.,2.)-- (6.,2.);
\draw [line width=1.pt] (2.,6.)-- (2.,1.);
\draw [line width=1.pt] (1.5,6.)-- (1.5,1.);
\draw [line width=1.pt] (1.,5.5)-- (6.,5.5);
\draw [line width=1.pt] (1.,1.5)-- (6.,1.5);
\draw [line width=1.pt] (5.5,6.)-- (5.5,1.);
\draw [line width=2.pt] (3,5.5)-- (3.,2.);
\draw [line width=2.pt] (1.5,4)-- (5.,4.);
\draw [line width=2.pt] (2.,3.)-- (5.5,3);
\draw [line width=2.pt] (4,1.5)-- (4.,5.);
\draw [line width=2.pt] (5.,1.5)-- (5.,5.5);
\begin{scriptsize}
\draw [fill=black] (3,5.5) circle (2.0pt);
\draw[color=black] (3.13,5.65) node {$v_1$};
\draw [fill=black] (5,5.5) circle (2.0pt);
\draw[color=black] (5.13,5.65) node {$v_9$};
\draw [fill=black] (3.,2.) circle (2.0pt);
\draw[color=black] (3.14,2.15) node {$v_7$};
\draw [fill=black] (1.5,4) circle (2.0pt);
\draw[color=black] (1.64,4.15) node {$v_3$};
\draw [fill=black] (5.,4.) circle (2.0pt);
\draw[color=black] (5.14,4.15) node {$v_4$};
\draw [fill=black] (2.,3.) circle (2.0pt);
\draw[color=black] (2.13,3.15) node {$v_5$};
\draw [fill=black] (5.5,3) circle (2.0pt);
\draw[color=black] (5.65,3.15) node {$v_6$};
\draw [fill=black] (4,1.5) circle (2.0pt);
\draw[color=black] (4.13,1.34) node {$v_8$};
\draw [fill=black] (5,1.5) circle (2.0pt);
\draw[color=black] (5.13,1.34) node {$v_{10}$};
\draw [fill=black] (4.,5.) circle (2.0pt);
\draw[color=black] (4.13,5.15) node {$v_2$};

\draw [fill=black] (3.,4.) circle (2.0pt);
\draw[color=black] (3.2,4.15) node {$v_{11}$};
\draw [fill=black] (3.,3.) circle (2.0pt);
\draw[color=black] (3.2,3.15) node {$v_{12}$};
\draw [fill=black] (4.,4.) circle (2.0pt);
\draw[color=black] (4.2,4.15) node {$v_{13}$};
\draw [fill=black] (4.,3.) circle (2.0pt);
\draw[color=black] (4.2,3.15) node {$v_{14}$};
\draw [fill=black] (5.,3.) circle (2.0pt);
\draw[color=black] (5.2,3.15) node {$v_{15}$};
\draw [fill=black] (3.,5.) circle (2.0pt);
\draw[color=black] (3.2,5.15) node {$v_{16}$};
\draw [fill=black] (2.,4.) circle (2.0pt);
\draw[color=black] (2.2,4.15) node {$v_{17}$};
\draw [fill=black] (4.,2.) circle (2.0pt);
\draw[color=black] (4.2,2.15) node {$v_{18}$};
\draw [fill=black] (5.,2.) circle (2.0pt);
\draw[color=black] (5.2,2.15) node {$v_{19}$};
\draw [fill=black] (5.,5.) circle (2.0pt);
\draw[color=black] (5.2,5.15) node {$v_{20}$};
\end{scriptsize}
\end{tikzpicture}
\caption{\label{fig: a general T-mesh}An example of a general T-mesh.}
\end{figure}

\begin{figure}[htbp]
\centering
\begin{minipage}[c]{0.48\textwidth}
\centering
\begin{tikzpicture}[line cap=round,line join=round,>=triangle 45,x=1.0cm,y=1.0cm,scale=1]

\draw [line width=1.pt] (3,5.5)-- (3.,2.);       
\draw [line width=1.pt] (1.5,4)-- (5.,4.);       
\draw [line width=1.pt] (2.,3.)-- (5.5,3);       
\draw [line width=1.pt] (4,1.5)-- (4.,5.);       
\draw [line width=1.pt] (5.,1.5)-- (5.,5.5);     

\begin{scriptsize}
\draw [fill=black] (3,5.5) circle (2.0pt);
\draw[color=black] (3.2,5.65) node {$v_1$};
\draw [fill=black] (5,5.5) circle (2.0pt);
\draw[color=black] (5.2,5.65) node {$v_9$};
\draw [fill=black] (3.,2.) circle (2.0pt);
\draw[color=black] (3.2,2.1) node {$v_7$};
\draw [fill=black] (1.5,4) circle (2.0pt);
\draw[color=black] (1.54,4.15) node {$v_3$};
\draw [fill=black] (5.,4.) circle (2.0pt);
\draw[color=black] (5.2,4.15) node {$v_4$};
\draw [fill=black] (2.,3.) circle (2.0pt);
\draw[color=black] (2.2,3.15) node {$v_5$};
\draw [fill=black] (5.5,3) circle (2.0pt);
\draw[color=black] (5.65,3.15) node {$v_6$};
\draw [fill=black] (4,1.5) circle (2.0pt);
\draw[color=black] (4.2,1.34) node {$v_8$};
\draw [fill=black] (5,1.5) circle (2.0pt);
\draw[color=black] (5.2,1.34) node {$v_{10}$};
\draw [fill=black] (4.,5.) circle (2.0pt);
\draw[color=black] (4.2,5.15) node {$v_2$};

\draw [fill=black] (3.,4.) circle (2.0pt);
\draw[color=black] (3.2,4.15) node {$v_{11}$};
\draw [fill=black] (3.,3.) circle (2.0pt);
\draw[color=black] (3.2,3.15) node {$v_{12}$};
\draw [fill=black] (4.,4.) circle (2.0pt);
\draw[color=black] (4.2,4.15) node {$v_{13}$};
\draw [fill=black] (4.,3.) circle (2.0pt);
\draw[color=black] (4.2,3.15) node {$v_{14}$};
\draw [fill=black] (5.,3.) circle (2.0pt);
\draw[color=black] (5.2,3.15) node {$v_{15}$};
\draw [fill=black] (3.,5.) circle (2.0pt);
\draw[color=black] (3.2,5.15) node {$v_{16}$};
\draw [fill=black] (2.,4.) circle (2.0pt);
\draw[color=black] (2.2,4.15) node {$v_{17}$};
\draw [fill=black] (4.,2.) circle (2.0pt);
\draw[color=black] (4.2,2.15) node {$v_{18}$};
\draw [fill=black] (5.,2.) circle (2.0pt);
\draw[color=black] (5.2,2.15) node {$v_{19}$};
\draw [fill=black] (5.,5.) circle (2.0pt);
\draw[color=black] (5.2,5.15) node {$v_{20}$};
\end{scriptsize}
\end{tikzpicture}
\\ \vspace{0.2cm}
\centering{(a) T-connected component $T(\mathscr{T})$}
\end{minipage}
\hfill 
\begin{minipage}[c]{0.48\textwidth}
\centering
\begin{tikzpicture}[line cap=round,line join=round,>=triangle 45,x=1.0cm,y=1.0cm,scale=1]

\draw [line width=1.pt] (3,5.5)-- (3.,2.);       
\draw [line width=1.pt] (1.5,4)-- (5.,4.);       
\draw [line width=1.pt] (2.,3.)-- (5.5,3);       
\draw [line width=1.pt] (4,1.5)-- (4.,5.);       

\begin{scriptsize}
\draw [fill=black] (3,5.5) circle (2.0pt);
\draw[color=black] (3.2,5.65) node {$v_1$};
\draw [fill=black] (3.,2.) circle (2.0pt);
\draw[color=black] (3.2,2.1) node {$v_7$};
\draw [fill=black] (1.5,4) circle (2.0pt);
\draw[color=black] (1.54,4.15) node {$v_3$};
\draw [fill=black] (5.,4.) circle (2.0pt);
\draw[color=black] (5.1,4.15) node {$v_4$};
\draw [fill=black] (2.,3.) circle (2.0pt);
\draw[color=black] (2.1,3.15) node {$v_5$};
\draw [fill=black] (5.5,3) circle (2.0pt);
\draw[color=black] (5.65,3.15) node {$v_6$};
\draw [fill=black] (4,1.5) circle (2.0pt);
\draw[color=black] (4.2,1.34) node {$v_8$};
\draw [fill=black] (4.,5.) circle (2.0pt);
\draw[color=black] (4.2,5.15) node {$v_2$};

\draw [fill=black] (3.,4.) circle (2.0pt);
\draw[color=black] (3.2,4.15) node {$v_{11}$};
\draw [fill=black] (3.,3.) circle (2.0pt);
\draw[color=black] (3.2,3.15) node {$v_{12}$};
\draw [fill=black] (4.,4.) circle (2.0pt);
\draw[color=black] (4.2,4.15) node {$v_{13}$};
\draw [fill=black] (4.,3.) circle (2.0pt);
\draw[color=black] (4.2,3.15) node {$v_{14}$};
\draw [fill=black] (5.,3.) circle (2.0pt);
\draw[color=black] (5.1,3.15) node {$v_{15}$};
\draw [fill=black] (3.,5.) circle (2.0pt);
\draw[color=black] (3.2,5.15) node {$v_{16}$};
\draw [fill=black] (2.,4.) circle (2.0pt);
\draw[color=black] (2.1,4.15) node {$v_{17}$};
\draw [fill=black] (4.,2.) circle (2.0pt);
\draw[color=black] (4.2,2.15) node {$v_{18}$};
\end{scriptsize}
\end{tikzpicture}
\\ \vspace{0.1cm}
\centering{(b) CNDC of $T(\mathscr{T})$}
\\ \vspace{0.5cm} 

\begin{tikzpicture}[line cap=round,line join=round,>=triangle 45,x=1.0cm,y=1.0cm,scale=1]
\draw [line width=1.pt] (5.,1.5)-- (5.,5.5);     

\begin{scriptsize}
\draw [fill=black] (5,5.5) circle (2.0pt);
\draw[color=black] (5.13,5.65) node {$v_9$};
\draw [fill=black] (5,1.5) circle (2.0pt);
\draw[color=black] (5.13,1.24) node {$v_{10}$};

\draw [fill=black] (5.,3.) circle (2.0pt);
\draw[color=black] (5.3,3.15) node {$v_{15}$};
\draw [fill=black] (5.,2.) circle (2.0pt);
\draw[color=black] (5.3,2.15) node {$v_{19}$};
\draw [fill=black] (5.,5.) circle (2.0pt);
\draw[color=black] (5.3,5.15) node {$v_{20}$};
\end{scriptsize}
\end{tikzpicture}
\\ \vspace{0.1cm}
\centering{(c) Diagonalizable component of $T(\mathscr{T})$}
\end{minipage}

\caption{\label{fig:T-mesh decomposition framework}The complete partition of $T(\mathscr{T})$.}
\end{figure}
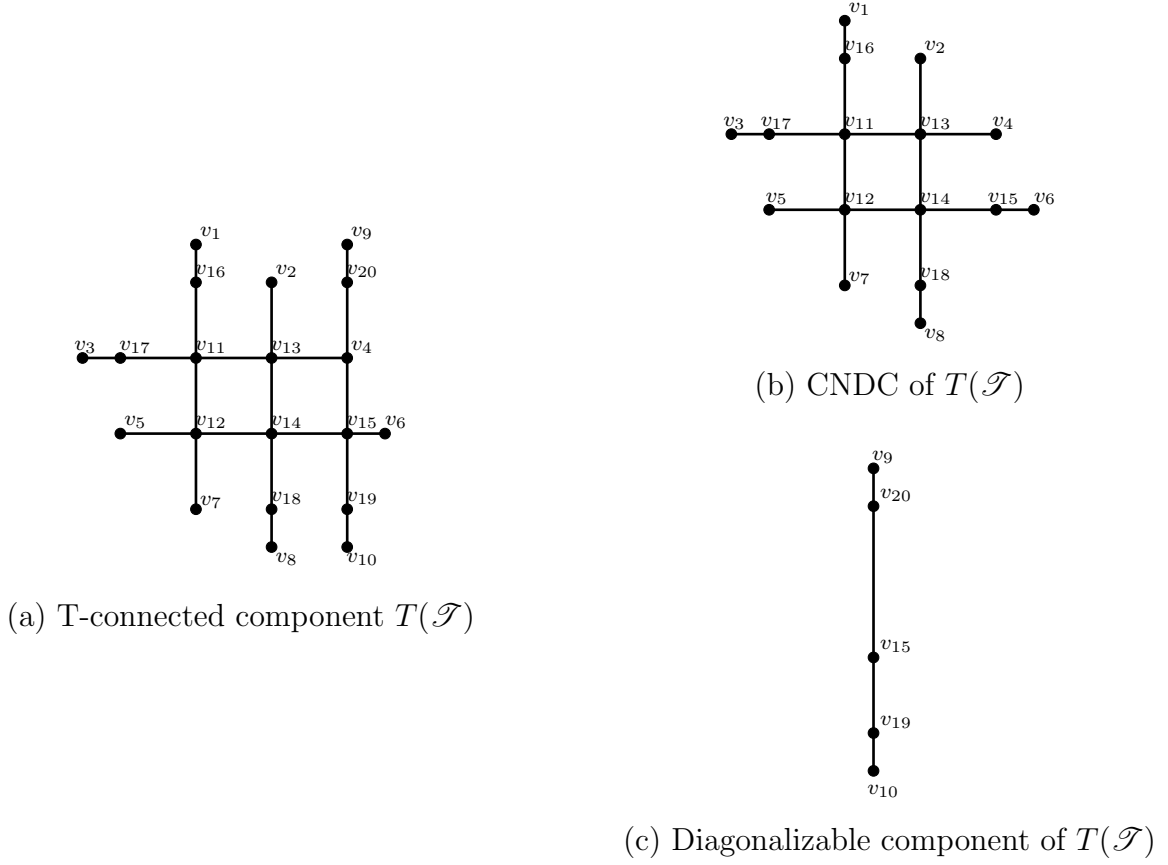

According to Proposition~\ref{prop:uniqueness_complete_partition}, the complete partition allows us to simplify the dimension calculation. Since the diagonalizable component does not introduce dimensional instability, the instability depends entirely on the completely non-diagonalizable component (CNDC). This relationship is formalized in the following theorem from \cite{huang2024stability}.

\begin{theorem}[\cite{huang2024stability}]\label{thm:cndc_dim}
Given a T-mesh $\mathscr{T}$, let $M\left(T(\mathscr{T})\right)$ be the global conformality matrix of the T-connected component $T(\mathscr{T})$, and let $\{T_1, T_2\}$ be the complete partition of $T(\mathscr{T})$ where the CNDC is denoted as $T_1 = \{l_1, l_2, \ldots, l_s\}$ ($s \le t$). Then, the dimension of the spline space $S_d(\mathscr{T})$ is:
\begin{equation}\label{dimformulanew}
\dim S_{d}(\mathscr{T})=(d+1)^2+(c+s-t)(d+1)+n_v-\mathrm{rank}(M_1),
\end{equation}
where $c$ is the number of all cross-cut edges, $n_v$ is the number of all interior vertices, and $M_1$ is the conformality matrix corresponding to the CNDC of $T(\mathscr{T})$.
\end{theorem}

Next, we provide an example to illustrate the conformality matrix corresponding to the CNDC of $T(\mathscr{T})$.

\begin{example}[\cite{huang2024stability}]\label{ex:conformality_matrix}
Continuing from Example \ref{ex:complete_partition}, we explicitly illustrate the concrete structure of the conformality matrix $M_1$ corresponding to the CNDC of $T(\mathscr{T})$. Let $T_1$ denote the CNDC of $T(\mathscr{T})$ shown in Fig.~\ref{fig:T-mesh decomposition framework}(b), which consists of four mutually intersecting T $l$-edges: $v_1v_7, v_3v_4, v_2v_8$, and $v_5v_6$. 

To construct the conformality matrix systematically, the ordering of the T $l$-edges is prescribed as $v_1v_7 \succ v_3v_4 \succ v_2v_8 \succ v_5v_6$. Along each individual T $l$-edge, the local vertices are ordered sequentially as follows:
\begin{itemize}
    \item For $v_1v_7$: $v_1 \succ v_7 \succ v_{16} \succ v_{11} \succ v_{12}$
    \item For $v_3v_4$: $v_{11} \succ v_3 \succ v_{17} \succ v_4 \succ v_{13}$
    \item For $v_2v_8$: $v_{13} \succ v_8 \succ v_{18} \succ v_2 \succ v_{14}$
    \item For $v_5v_6$: $v_{12} \succ v_{14} \succ v_5 \succ v_{15} \succ v_6$
\end{itemize}
Since the vertex smoothing cofactors under the highest order of smoothness setting are in one-to-one correspondence with the vertices, each vertex $v_i$ uniquely determines a cofactor $\delta_i$. By following the prescribed edge priority and collecting the vertices sequentially without duplication, we obtain the global column ordering for these vertex smoothing cofactors. The resulting sequence of the 16 cofactor-associated columns is given by 

$$\delta_1, \delta_7, \delta_{16}, \delta_{11}, \delta_{12}, \delta_3, \delta_{17}, \delta_4, \delta_{13}, \delta_8, \delta_{18}, \delta_2, \delta_{14}, \delta_5, \delta_{15}, \delta_6.$$

Under the above global ordering, the conformality matrix $M_1$ is formulated as a $16 \times 16$ matrix. To explicitly indicate the correspondence between the columns and the vertex smoothing cofactors, the headers $\delta_i$ are positioned outside and directly above the matrix columns as follows:

\[
M_1 = \bordermatrix{
    & \delta_1 & \delta_7 & \delta_{16} & \delta_{11} & \delta_{12} & \delta_3 & \delta_{17} & \delta_4 & \delta_{13} & \delta_8 & \delta_{18} & \delta_2 & \delta_{14} & \delta_5 & \delta_{15} & \delta_6 \cr
    & 1 & 1 & 1 & 1 & 1 & 0 & 0 & 0 & 0 & 0 & 0 & 0 & 0 & 0 & 0 & 0 \cr
    & t_7 & t_3 & t_6 & t_5 & t_4 & 0 & 0 & 0 & 0 & 0 & 0 & 0 & 0 & 0 & 0 & 0 \cr
    & t_7^2 & t_3^2 & t_6^2 & t_5^2 & t_4^2 & 0 & 0 & 0 & 0 & 0 & 0 & 0 & 0 & 0 & 0 & 0 \cr
    & t_7^3 & t_3^3 & t_6^3 & t_5^3 & t_4^3 & 0 & 0 & 0 & 0 & 0 & 0 & 0 & 0 & 0 & 0 & 0 \cr
    & 0 & 0 & 0 & 1 & 0 & 1 & 1 & 1 & 1 & 0 & 0 & 0 & 0 & 0 & 0 & 0 \cr
    & 0 & 0 & 0 & s_4 & 0 & s_2 & s_3 & s_6 & s_5 & 0 & 0 & 0 & 0 & 0 & 0 & 0 \cr
    & 0 & 0 & 0 & s_4^2 & 0 & s_2^2 & s_3^2 & s_6^2 & s_5^2 & 0 & 0 & 0 & 0 & 0 & 0 & 0 \cr
    & 0 & 0 & 0 & s_4^3 & 0 & s_2^3 & s_3^3 & s_6^3 & s_5^3 & 0 & 0 & 0 & 0 & 0 & 0 & 0 \cr
    & 0 & 0 & 0 & 0 & 0 & 0 & 0 & 0 & 1 & 1 & 1 & 1 & 1 & 0 & 0 & 0 \cr
    & 0 & 0 & 0 & 0 & 0 & 0 & 0 & 0 & t_5 & t_2 & t_3 & t_6 & t_4 & 0 & 0 & 0 \cr
    & 0 & 0 & 0 & 0 & 0 & 0 & 0 & 0 & t_5^2 & t_2^2 & t_3^2 & t_6^2 & t_4^2 & 0 & 0 & 0 \cr
    & 0 & 0 & 0 & 0 & 0 & 0 & 0 & 0 & t_5^3 & t_2^3 & t_3^3 & t_6^3 & t_4^3 & 0 & 0 & 0 \cr
    & 0 & 0 & 0 & 0 & 1 & 0 & 0 & 0 & 0 & 0 & 0 & 0 & 1 & 1 & 1 & 1 \cr
    & 0 & 0 & 0 & 0 & s_4 & 0 & 0 & 0 & 0 & 0 & 0 & 0 & s_5 & s_3 & s_6 & s_7 \cr
    & 0 & 0 & 0 & 0 & s_4^2 & 0 & 0 & 0 & 0 & 0 & 0 & 0 & s_5^2 & s_3^2 & s_6^2 & s_7^2 \cr
    & 0 & 0 & 0 & 0 & s_4^3 & 0 & 0 & 0 & 0 & 0 & 0 & 0 & s_5^3 & s_3^3 & s_6^3 & s_7^3
}
\]

where $s_i$ and $t_j$ represent the $x$- and $y$-coordinates of the grid lines, respectively.

By applying elementary row operations to eliminate the rows and columns corresponding to the mono-vertices according to the reduction technique in \cite{huang2024stability}, the conformality matrix $M_1$ can be reduced to the conformality matrix $M_1^{\text{mult}}$ corresponding to the multi-vertices within the CNDC:
After eliminating the mono-vertex smoothing cofactors, the reduced conformality matrix $M_1^{\text{mult}}$ involves only the multi-vertex smoothing cofactors. Let $\delta_{11}, \delta_{12}, \delta_{13}, \delta_{14}$ denote the cofactors associated with these four multi-vertices. To explicitly indicate the correspondence between the columns and the multi-vertex smoothing cofactors, the headers are positioned outside and directly above the matrix columns as follows:

\[
M_1^{\text{mult}} = \bordermatrix{
    & \delta_{11} & \delta_{12} & \delta_{13} & \delta_{14} \cr
    & 1 & f_{3,6,7}(t_{4,5}) & 0 & 0 \cr
    & f_{2,3,6}(s_{4,5}) & 0 & 1 & 0 \cr
    & 0 & 0 & f_{2,3,6}(t_{5,4}) & 1 \cr
    & 0 & f_{3,6,7}(s_{4,5}) & 0 & 1
}
\]
where the functions $f_{i,j,k}(x_{p,q})$ are defined as:
\begin{equation*}
f_{i,j,k}(x_{p,q})=\frac{(x_p-x_i)(x_p-x_j)(x_p-x_k)}{(x_q-x_i)(x_q-x_j)(x_q-x_k)}, \quad x \in \{s, t\}.
\end{equation*}
Through further elementary transformations, $M_1^{\text{mult}}$ reduces to
$$
\begin{pmatrix}
1 & 0 & 0 & 0 \\ 0 & 1 & 0 & 0 \\
0 & 0 & 1 & 0 \\
0 & 0 & 0 & \det(M_1^{\text{mult}})
\end{pmatrix}
$$
with $\det({M_1^{\text{mult}}})=
f_{3,6,7}(s_{4,5})-f_{2,3,6}(t_{5,4})f_{2,3,6}(s_{4,5})f_{3,6,7}(t_{4,5})$.
Thus 
$$rank(M_1)=12+rank({M_1^{\text{mult}}})=\begin{cases} 15 & \det({M_1^{\text{mult}}})=0 \\
16 & \det({M_1^{\text{mult}}})\not=0
\end{cases}$$
For a more detailed discussion regarding this example, one may refer to Example 4.1 in \cite{huang2024stability}.
\end{example}

Furthermore, by generalizing the reduction technique demonstrated in Example \ref{ex:conformality_matrix}, the evaluation of the massive coupled system can be further simplified. We can explicitly determine the spline space dimension by relying strictly on the condensed multi-vertex conformality matrix, thereby isolating the core topological influence of the CNDC, as stated in the following theorem.

\begin{theorem}[\cite{huang2024stability}]\label{thm:cndc_dim_mult}
Under the same assumptions as in Theorem \ref{thm:cndc_dim}, let $n_{\textit{CNDC}}$ denote the number of vertices on the CNDC, and let $m$ be the number of multi-vertices within the CNDC. By applying elementary operations to eliminate the smoothing cofactors associated with the $n_{\textit{CNDC}} - m$ mono-vertices, the conformality matrix $M_1$ can be reduced to a matrix $M_1^{\text{mult}}\in\mathbb{R}^{(s(d+1)-n_{\textit{CNDC}}+m)\times m}$ that corresponds exclusively to the multi-vertices. Then, the dimension of the spline space $S_d(\mathscr{T})$ satisfies:
\begin{equation}
\dim S_{d}(\mathscr{T}) = (d+1)^2 + (c+s-t)(d+1) + n_v - n_{\textit{CNDC}} + m - \mathrm{rank}(M_1^{\text{mult}}).
\end{equation}
\end{theorem}

Based on the rank estimation of $M_1^{\text{mult}}$ in Theorem~\ref{thm:cndc_dim_mult}, we can establish explicit upper and lower bounds for the dimension of the spline space.

\begin{corollary}\label{cor:dim_bounds}
Let $M_1^{\text{mult}} \in \mathbb{R}^{(s(d+1) - n_{\textit{CNDC}} + m) \times m}$ be the reduced multi-vertex conformality matrix defined in Theorem~\ref{thm:cndc_dim_mult}. The dimension of the spline space $S_d(\mathscr{T})$ is bounded from below by:
\begin{equation}\label{eq:dim_bounds_formula}
L(\mathscr{T})\le \dim S_{d}(\mathscr{T}),
\end{equation}
where the lower bound $L(\mathscr{T})$ is given by:
\begin{equation}
L(\mathscr{T}) = (d+1)^2 + (c+s-t)(d+1) + n_v - n_{\textit{CNDC}} + \big(n_{\textit{CNDC}} - s(d+1)\big)_+
\end{equation}
and $(\cdot)_+ = \max(\cdot, 0)$ denotes the truncation operator.
\end{corollary}

\begin{proof}
Since $M_1^{\text{mult}}$ is a matrix of size $(s(d+1) - n_{\textit{CNDC}} + m) \times m$, its rank is fundamentally bounded from above by the minimum of its row and column dimensions:
\begin{equation}\label{eq:rank_upper}
\mathrm{rank}(M_1^{\text{mult}}) \le \min\big(s(d+1) - n_{\textit{CNDC}} + m, \, m\big).
\end{equation}
The minimum dimension of the spline space occurs when $\mathrm{rank}(M_1^{\text{mult}})$ reaches this upper bound. Substituting (\ref{eq:rank_upper}) into the dimension formula in Theorem~\ref{thm:cndc_dim_mult} yields the lower bound:
\begin{align*}
\dim S_{d}(\mathscr{T}) &\ge (d+1)^2 + (c+s-t)(d+1) + n_v - n_{\textit{CNDC}} + m - \min\big(s(d+1) - n_{\textit{CNDC}} + m, \, m\big) \\
&= (d+1)^2 + (c+s-t)(d+1) + n_v - n_{\textit{CNDC}} + \max\big(m - (s(d+1) - n_{\textit{CNDC}} + m), \, m - m\big) \\
&= (d+1)^2 + (c+s-t)(d+1) + n_v - n_{\textit{CNDC}} + \max\big(n_{\textit{CNDC}} - s(d+1), \, 0\big) \\
&= (d+1)^2 + (c+s-t)(d+1) + n_v - n_{\textit{CNDC}} + \big(n_{\textit{CNDC}} - s(d+1)\big)_+.
\end{align*}
This completes the proof.
\end{proof}
\section{The decoupling technique for conformality systems}
In this section, we introduce a decoupling technique specifically designed to resolve the tightly coupled conformality conditions of multi-vertices within the CNDC of $T(\mathscr{T})$. As established in Theorem~\ref{thm:cndc_dim}, characterizing the dimension of the spline space reduces to evaluating the conformality conditions over the CNDC. Moreover, the rank of this conformality matrix ultimately depends on the conformality conditions associated with the corresponding multi-vertex smoothing cofactors.

However, directly evaluating the rank of the matrix $M_1^{\text{mult}}$ by following the approach implied in Theorem~\ref{thm:cndc_dim} presents several significant theoretical and computational limitations in practice:
\begin{enumerate}
    \item The construction of the monolithic conformality matrix $M_1$ depends heavily on the specific global ordering of the T $l$-edges and vertices, making it highly intractable to formulate directly for a CNDC with a complex topological structure.
    \item Obtaining the reduced matrix $M_1^{\text{mult}}$ requires performing intensive elementary row operations to eliminate the rows and columns associated with the mono-vertices, which leads to an expensive computational cost as the matrix size scales.
    \item The direct computation can only be conducted on a case-by-case basis for specific mesh configurations. It remains extremely difficult to provide a general form or establish a uniform solving approach using this direct method.
\end{enumerate}

To circumvent these bottlenecks and address the lack of a generalized approach, we introduce the decoupling technique in this section. Structurally, the complexity stems from the severe algebraic coupling at the multi-vertices. Specifically, as illustrated by the conformality matrix $M_1$ in Example~\ref{ex:conformality_matrix}, the multi-vertex cofactors $\delta_{11}, \delta_{12}, \delta_{13}$, and $\delta_{14}$ associated with the multi-vertices $v_{11}, v_{12}, v_{13}$, and $v_{14}$ correspond precisely to the 4th, 5th, 9th, and 13th columns of $M_1$, respectively. These columns inherently exhibit a strong coupling nature, characterized by the presence of multiple distinct sets of non-zero entries that span across different row blocks (i.e., representing separate horizontal and vertical edge subsystems). This structural overlap intertwines the otherwise independent localized equations, thereby destroying the block-diagonal property of the global system and preventing parallel evaluation. Geometrically, a multi-vertex cofactor must simultaneously satisfy the smooth continuity conditions imposed by both a horizontal T $l$-edge and a vertical T $l$-edge. To break this mutual dependency, the proposed technique splits the dual geometric roles embedded within a single multi-vertex cofactor by introducing auxiliary variables, thereby allowing the horizontal and vertical smoothing constraints to be analyzed independently in a localized manner.

Before formally presenting the main theorem and the general algorithm, we first introduce the formal definition of \textbf{decoupled vertex cofactors}. Based on this concept, the core mechanics of the decoupling method are illustrated using the same mesh configuration as in Example~\ref{ex:conformality_matrix}. Subsequently, we establish the general computational framework and present the main results regarding the spline space dimension.

According to Proposition~\ref{prop:uniqueness_complete_partition}, the CNDC of a given T-mesh $\mathscr{T}$ is uniquely determined. Consequently, the number of the constituent T $l$-edges $s$, the number of the vertices $n_{\textit{CNDC}}$, and the number of the multi-vertices $m$ are all uniquely determined. This ensures that the following concept is well-defined.

\begin{definition}[Decoupled Vertex Cofactors]
Given a T-mesh $\mathscr{T}$ and the corresponding spline space $S_d(\mathscr{T})$, let $T_1 = \{l_1, l_2, \ldots, l_s\}$ be the CNDC of $\mathscr{T}$ containing $m$ interior multi-vertices $\{v_1, v_2, \ldots, v_m\}$. For each multi-vertex $v_i$ ($i=1,2,\ldots,m$), let $v_i = l_{j_i} \cap l_{k_i}$, where $l_{j_i} \in T_1$ is a horizontal T $l$-edge and $l_{k_i} \in T_1$ is a vertical T $l$-edge. The \textbf{decoupled vertex cofactors} for $v_i$ are defined as a pair of  variables $(\delta_i^{h}, \delta_i^{v}) \in \mathbb{R}^2$, where:
\begin{itemize}
    \item $\delta_i^{h}$ is the horizontal decoupled vertex cofactor associated with the horizontal T $l$-edge $l_{j_i}$, satisfying the local conformality condition of the form \eqref{gcc} along $l_{j_i}$;
    \item $\delta_i^{v}$ is the vertical decoupled vertex cofactor associated with the vertical T $l$-edge $l_{k_i}$, satisfying the local conformality condition of the form \eqref{gcc} along $l_{k_i}$.
\end{itemize}
\end{definition}

Before presenting a concrete example to illustrate how decoupled vertex cofactors are used in actual computations, we first introduce the following lemma.

\begin{lemma}[\cite{huang2024stability}]
\label{lemma:vander_rref}
Let $I = \{i_1, i_2, \ldots, i_k\} \subset \mathbb{N}$ be an index set of $k$ distinct elements. Let $V_k^d$ be the $(d+1) \times k$ Vandermonde matrix generated by distinct nodes $\{s_j\}_{j \in I} \subset \mathbb{R}$, where the columns are ordered according to the sequence of $I$. Suppose $k > d+1$, and let $I_0 = \{i_1, i_2, \ldots, i_{d+1}\}$ denote the subset of the first $d+1$ indices. Then, the reduced row echelon form of $V_k^d$ is simplified as $\begin{pmatrix} I_{d+1} & S \end{pmatrix}$, where $S$ is a $(d+1) \times (k-d-1)$ matrix whose columns are indexed by $j \in I \setminus I_0$, given by
\[
    S = \begin{pmatrix}
        g_{i_1}(s_{i_{d+2}}) & \cdots & g_{i_1}(s_{i_k}) \\
        \vdots & \ddots & \vdots \\
        g_{i_{d+1}}(s_{i_{d+2}}) & \cdots & g_{i_{d+1}}(s_{i_k})
    \end{pmatrix},
\]
and the functions $g_i$ are defined by the Lagrange basis polynomials associated with the reference set $I_0$:
\[
    g_p(x) = \prod_{j \in I_0, j \neq p} \frac{x-s_j}{s_p-s_j}, \quad p \in I_0.
\]
\end{lemma}

\bigskip

Below, we revisit the mesh configuration in Example~\ref{ex:conformality_matrix} to demonstrate how to evaluate the rank of the conformality matrix associated with CNDC using the decoupled vertex cofactors and Lemma~\ref{lemma:vander_rref}.

\begin{example}\label{ex:decoupled_conformality_matrix}
Continuing from Example~\ref{ex:conformality_matrix}, we expand the column space from $16$ to $20$ by replacing the monolithic multi-vertex cofactors $\delta_{11}, \delta_{12}, \delta_{13},$ and $\delta_{14}$ with their decoupled horizontal and vertical counterparts. To optimize the local basis structure of block $4$, we arrange the expanded vertex cofactor vector as:
\[
\boldsymbol{\delta}_1 = (\delta_1, \delta_7, \delta_{16}, \delta_{11}^h, \delta_{12}^h, \delta_{11}^v, \delta_3, \delta_{17}, \delta_4, \delta_{13}^v, \delta_{13}^h, \delta_8, \delta_{18}, \delta_2, \delta_{14}^h, \delta_{14}^v, \delta_5, \delta_{15}, \delta_6, \delta_{12}^v)^T \in \mathbb{R}^{20}.
\]
Under this formulation, the modified conformality system for this CNDC can be expressed compactly as the homogeneous linear system $K_1 \boldsymbol{\delta}_1 = 0$. The $16 \times 20$ block-diagonal matrix $K_1$ natively preserves the rows according to the edge priority $v_1v_7 \succ v_3v_4 \succ v_2v_8 \succ v_5v_6$, which is explicitly given as follows:
\[
\bordermatrix{
    & \delta_1 & \delta_7 & \delta_{16} & \delta_{11}^h & \delta_{12}^h & \delta_{11}^v & \delta_3 & \delta_{17} & \delta_4 & \delta_{13}^v & \delta_{13}^h & \delta_8 & \delta_{18} & \delta_2 & \delta_{14}^h & \delta_{14}^v & \delta_5 & \delta_{15} & \delta_6 & \delta_{12}^v \cr
    & 1 & 1 & 1 & 1 & 1 & 0 & 0 & 0 & 0 & 0 & 0 & 0 & 0 & 0 & 0 & 0 & 0 & 0 & 0 & 0 \cr
    & t_7 & t_3 & t_6 & t_5 & t_4 & 0 & 0 & 0 & 0 & 0 & 0 & 0 & 0 & 0 & 0 & 0 & 0 & 0 & 0 & 0 \cr
    & t_7^2 & t_3^2 & t_6^2 & t_5^2 & t_4^2 & 0 & 0 & 0 & 0 & 0 & 0 & 0 & 0 & 0 & 0 & 0 & 0 & 0 & 0 & 0 \cr
    & t_7^3 & t_3^3 & t_6^3 & t_5^3 & t_4^3 & 0 & 0 & 0 & 0 & 0 & 0 & 0 & 0 & 0 & 0 & 0 & 0 & 0 & 0 & 0 \cr
    & 0 & 0 & 0 & 0 & 0 & 1 & 1 & 1 & 1 & 1 & 0 & 0 & 0 & 0 & 0 & 0 & 0 & 0 & 0 & 0 \cr
    & 0 & 0 & 0 & 0 & 0 & s_4 & s_2 & s_3 & s_6 & s_5 & 0 & 0 & 0 & 0 & 0 & 0 & 0 & 0 & 0 & 0 \cr
    & 0 & 0 & 0 & 0 & 0 & s_4^2 & s_2^2 & s_3^2 & s_6^2 & s_5^2 & 0 & 0 & 0 & 0 & 0 & 0 & 0 & 0 & 0 & 0 \cr
    & 0 & 0 & 0 & 0 & 0 & s_4^3 & s_2^3 & s_3^3 & s_6^3 & s_5^3 & 0 & 0 & 0 & 0 & 0 & 0 & 0 & 0 & 0 & 0 \cr
    & 0 & 0 & 0 & 0 & 0 & 0 & 0 & 0 & 0 & 0 & 1 & 1 & 1 & 1 & 1 & 0 & 0 & 0 & 0 & 0 \cr
    & 0 & 0 & 0 & 0 & 0 & 0 & 0 & 0 & 0 & 0 & t_5 & t_2 & t_3 & t_6 & t_4 & 0 & 0 & 0 & 0 & 0 \cr
    & 0 & 0 & 0 & 0 & 0 & 0 & 0 & 0 & 0 & 0 & t_5^2 & t_2^2 & t_3^2 & t_6^2 & t_4^2 & 0 & 0 & 0 & 0 & 0 \cr
    & 0 & 0 & 0 & 0 & 0 & 0 & 0 & 0 & 0 & 0 & t_5^3 & t_2^3 & t_3^3 & t_6^3 & t_4^3 & 0 & 0 & 0 & 0 & 0 \cr
    & 0 & 0 & 0 & 0 & 0 & 0 & 0 & 0 & 0 & 0 & 0 & 0 & 0 & 0 & 0 & 1 & 1 & 1 & 1 & 1 \cr
    & 0 & 0 & 0 & 0 & 0 & 0 & 0 & 0 & 0 & 0 & 0 & 0 & 0 & 0 & 0 & s_5 & s_3 & s_6 & s_7 & s_4 \cr
    & 0 & 0 & 0 & 0 & 0 & 0 & 0 & 0 & 0 & 0 & 0 & 0 & 0 & 0 & 0 & s_5^2 & s_3^2 & s_6^2 & s_7^2 & s_4^2 \cr
    & 0 & 0 & 0 & 0 & 0 & 0 & 0 & 0 & 0 & 0 & 0 & 0 & 0 & 0 & 0 & s_5^3 & s_3^3 & s_6^3 & s_7^3 & s_4^3
}
\]
Since the vertex cofactors $\delta_{11}, \delta_{12}, \delta_{13},$ and $\delta_{14}$ must simultaneously satisfy the conformality constraints from both their corresponding horizontal and vertical T $l$-edges, these decoupled counterparts are forced to satisfy the equality conditions $\delta_i^h = \delta_i^v$ for $i \in \{11, 12, 13, 14\}$. By omitting the columns that consist entirely of zeros and grouping the horizontal and vertical counterparts together, this coupling subsystem can be expressed compactly by retaining only the nonzero columns:
\[
K_2 = \bordermatrix{
    & \delta_{11}^h & \delta_{12}^h & \delta_{13}^h & \delta_{14}^h & \delta_{11}^v & \delta_{12}^v & \delta_{13}^v & \delta_{14}^v \cr
    v_{11} & 1 & 0 & 0 & 0 & -1 & 0 & 0 & 0 \cr
    v_{12} & 0 & 1 & 0 & 0 & 0 & -1 & 0 & 0 \cr
    v_{13} & 0 & 0 & 1 & 0 & 0 & 0 & -1 & 0 \cr
    v_{14} & 0 & 0 & 0 & 1 & 0 & 0 & 0 & -1
}.
\]
For notational simplicity, the zero columns corresponding to mono-vertex cofactors are omitted in the displayed matrix $K_2$. Whenever $K_1$ and $K_2$ are considered simultaneously, $K_2$ is understood to be zero-padded to the same $20$ columns as $K_1$.

As stated in Theorem~\ref{thm:cndc_dim}, solving the dimension of spline space over the CNDC is algebraically equivalent to finding the intersection of the kernels of $K_1$ and $K_2$, i.e., $\ker(K_1) \cap \ker(K_2)$.

\begin{itemize}
    \item \textbf{Traditional Approach:} One can first solve the system $K_2 \boldsymbol{\delta}_{\text{dec}} = 0$, where $\boldsymbol{\delta}_{\text{dec}}$ is the vector composed of the decoupled vertex cofactors. This system explicitly enforces the equality conditions $\delta_i^h = \delta_i^v$ for $i \in \{11, 12, 13, 14\}$. Substituting these relations into $K_1 \boldsymbol{\delta}_1 = 0$ eliminates the 4 redundant variables, reducing the 20-dimensional variable space back to 16. This substitution recovers the original system $M \boldsymbol{\delta} = 0$ defined in Example~\ref{ex:conformality_matrix}. Thus, computing the intersection $\ker(K_1) \cap \ker(K_2)$ is equivalent to evaluating the rank of $M$, which is consistent with the formulation in Theorem~\ref{thm:cndc_dim}.
    
    \item \textbf{Decoupling Framework:} Alternatively, one can exploit the block-diagonal structure of $K_1$ by computing its null space first. Since $K_1$ consists of four independent $4 \times 5$ Vandermonde-like submatrices, its total null space has a dimension of $20 - 16 = 4$. This localized structure allows the basis vectors to be solved analytically and independently on each individual edge block. Substituting these independent local basis expressions back into the coupling equation $K_2 \boldsymbol{\delta}_{\text{dec}} = 0$ then yields a highly condensed $4 \times 4$ linear system. This framework avoids dealing with the matrix $M$ directly, providing a parallelizable procedure to determine the spline dimension in Theorem~\ref{thm:cndc_dim}.
\end{itemize}

A massive advantage of this decoupled framework is that its localization significantly facilitates the theoretical analysis of the spline dimension over the CNDC. Furthermore, from a computational perspective, except for the initial layout of $K_1$, the entire solution process can be completely parallelized and evaluated block-by-block. Specifically, by setting $d=3$ and $k=5$ for each local block, Lemma~\ref{lemma:vander_rref} allows us to instantly obtain the reduced row echelon form for each of the four submatrices independently using localized, non-sequential index sets.

Let $\boldsymbol{\alpha} = (\alpha_1, \alpha_3, \alpha_2, \alpha_4)^T \in \mathbb{R}^4$ be the block-wise free parameters. We define the specific index sets $I^{(m)}$ and their respective basis reference subsets $I_0^{(m)}$ for each individual edge block $m \ (m=1,2,3,4)$ as follows:
\begin{itemize}
    \item \textbf{Block 1} ($v_1v_7$): $I^{(1)} = \{7, 3, 6, 5, 4\}$ with $I_0^{(1)} = \{7, 3, 6, 5\}$. The remaining node index is $4$. Lemma~\ref{lemma:vander_rref} yields:
    $$\delta_{11}^h = -g_5^{(1)}(t_4)\alpha_1, \quad \delta_{12}^h = \alpha_1.$$
    \item \textbf{Block 2} ($v_3v_4$): $I^{(2)} = \{4, 2, 3, 6, 5\}$ with $I_0^{(2)} = \{4, 2, 3, 6\}$. The remaining node index is $5$. Lemma~\ref{lemma:vander_rref} yields:
    $$\delta_{11}^v = -g_4^{(2)}(s_5)\alpha_2, \quad \delta_{13}^v = \alpha_2.$$
    \item \textbf{Block 3} ($v_2v_8$): $I^{(3)} = \{5, 2, 3, 6, 4\}$ with $I_0^{(3)} = \{5, 2, 3, 6\}$. The remaining node index is $4$. Lemma~\ref{lemma:vander_rref} yields:
    $$\delta_{13}^h = -g_5^{(3)}(t_4)\alpha_3, \quad \delta_{14}^h = \alpha_3.$$
    \item \textbf{Block 4} ($v_5v_6$): $I^{(4)} = \{5, 3, 6, 7, 4\}$ with $I_0^{(4)} = \{5, 3, 6, 7\}$. The remaining node index is $4$. Lemma~\ref{lemma:vander_rref} yields:
    $$\delta_{14}^v = -g_5^{(4)}(s_4)\alpha_4, \quad \delta_{12}^v = \alpha_4.$$
\end{itemize}
Here, the local Lagrange basis functions $g_p^{(m)}(x)$ ($p \in I_0^{(m)}$) are computed strictly via Lemma~\ref{lemma:vander_rref} using the geometric coordinates of the node subset $I_0^{(m)}$ for each respective block $m$.

These local basis solutions allow us to construct a multi-vertex basis matrix $Z \in \mathbb{R}^{8 \times 4}$ that isolates the evaluations of the null space basis vectors precisely at the multi-vertices. This matrix directly maps the block-wise free parameters $\boldsymbol{\alpha}$—which act as the coefficients of the basis of the null space for the decoupled homogeneous system $K_1 \boldsymbol{\delta}_1 = 0$—to the structured decoupled vertex cofactors $\boldsymbol{\delta}_{\text{dec}} = (\delta_{11}^h, \delta_{12}^h, \delta_{13}^h, \delta_{14}^h, \delta_{11}^v, \delta_{12}^v, \delta_{13}^v, \delta_{14}^v)^T$. By grouping the horizontal basis blocks above the vertical basis blocks, the matrix $Z$ is neatly structured as:
$$Z = \bordermatrix{
                & \alpha_1        & \alpha_3        & \alpha_2        & \alpha_4        \cr
\delta_{11}^h   & -g_5^{(1)}(t_4) & 0                & 0                & 0                \cr
\delta_{12}^h   & 1                & 0                & 0                & 0                \cr
\delta_{13}^h   & 0                & -g_5^{(3)}(t_4) & 0                & 0                \cr
\delta_{14}^h   & 0                & 1                & 0                & 0                \cr
\hline
\delta_{11}^v   & 0                & 0                & -g_4^{(2)}(s_5) & 0                \cr
\delta_{12}^v   & 0                & 0                & 0                & 1 \cr
\delta_{13}^v   & 0                & 0                & 1                & 0                \cr
\delta_{14}^v   & 0                & 0                & 0                & -g_5^{(4)}(s_4)
}.$$

Consequently, every vector in $\ker(K_1)$ has a unique coefficient vector $\boldsymbol{\alpha}$ with respect to the chosen block-wise basis, and its decoupled multi-vertex cofactors are given by $\boldsymbol{\delta}_{\text{dec}} = Z \boldsymbol{\alpha}$. To enforce the original constraints, these decoupled vertex cofactors must simultaneously satisfy the consistency requirements $K_2 \boldsymbol{\delta}_{\text{dec}} = \mathbf{0}$. Substituting this expression into the consistency conditions yields $(K_2 Z) \boldsymbol{\alpha} = \mathbf{0}$, which condenses the global consistency constraints into a highly reduced $4 \times 4$ homogeneous linear system governed by the condensed coupling matrix $K = K_2 Z$. This system evaluates explicitly to:
$$K \boldsymbol{\alpha} = 
\begin{pmatrix}
-g_5^{(1)}(t_4) & 0 & g_4^{(2)}(s_5) & 0 \\
1 & 0 & 0 & -1 \\
0 & -g_5^{(3)}(t_4) & -1 & 0 \\
0 & 1 & 0 & g_5^{(4)}(s_4)
\end{pmatrix}
\begin{pmatrix}
\alpha_1 \\ \alpha_3 \\ \alpha_2 \\ \alpha_4
\end{pmatrix} = 0.$$

Here, the entries are local Lagrange basis functions evaluated at the remaining nodes. For instance, the entry $g_5^{(1)}(t_4)$ is defined over the reference node set $I_0^{(1)} = \{7,3,6,5\}$ of the first block and is explicitly given by:
\[
g_5^{(1)}(t_4) = \frac{(t_4 - t_7)(t_4 - t_3)(t_4 - t_6)}{(t_5 - t_7)(t_5 - t_3)(t_5 - t_6)},
\]
with the other entries computed analogously over their respective geometric coordinate sets. 

The determinant of the updated coefficient matrix $K$ can be analytically evaluated as 
$$\det(K) = g_5^{(1)}(t_4) - g_4^{(2)}(s_5) \cdot g_5^{(3)}(t_4) \cdot g_5^{(4)}(s_4).$$ 
Notice that deleting the first row and the second column of $K$ yields an upper-triangular $3 \times 3$ submatrix (consisting of the entries $1$, $-1$, and $g_5^{(4)}(s_4)$ on its diagonal), implying $\text{rank}(K) \ge 3$ for every admissible choice of grid coordinates. Algebraically, solving the spline dimension over this CNDC via the kernel intersection $\ker(K_1) \cap \ker(K_2)$ is equivalent to finding $\ker(K)$, which directly governs the dimension as $$\dim(\ker(K_1) \cap \ker(K_2)) = 4 - \text{rank}(K).$$ 

Consequently, since the original conformality matrix $M_1$ defined in Example~\ref{ex:conformality_matrix} operates on a $16$-dimensional variable space, its rank is linked to the condensed system by  
$$\text{rank}(M_1) = 16 - \dim \ker(K) = 12 + \text{rank}(K).$$ 
This allows us to express the rank of the $16 \times 16$ conformality matrix $M_1$ explicitly as the following piecewise function:
\[
\text{rank}(M_1) = \begin{cases}
    15, & \text{if } g_5^{(1)}(t_4) - g_4^{(2)}(s_5) \cdot g_5^{(3)}(t_4) \cdot g_5^{(4)}(s_4)=0, \\
    16, & \text{if } g_5^{(1)}(t_4) - g_4^{(2)}(s_5) \cdot g_5^{(3)}(t_4) \cdot g_5^{(4)}(s_4) \neq 0.
\end{cases}
\]
Remarkably, this result demonstrates that the evaluation of the massive coupled system $M_1$ can be entirely avoided. It is easy to verify that the dimension instability condition 
$$g_5^{(1)}(t_4) - g_4^{(2)}(s_5) \cdot g_5^{(3)}(t_4) \cdot g_5^{(4)}(s_4) = 0$$ 
is algebraically equivalent to the condition derived in Example~\ref{ex:conformality_matrix}
$$f_{3,6,7}(s_{4,5})-f_{2,3,6}(t_{5,4})f_{2,3,6}(s_{4,5})f_{3,6,7}(t_{4,5})=0.$$ 
The non-trivial solutions of the condensed $4 \times 4$ system completely determine the parameters $\boldsymbol{\alpha}$, which can then be back-substituted to provide the exact smooth cofactors for this CNDC configuration.
\end{example}

The dramatic simplification observed in Example~\ref{ex:decoupled_conformality_matrix} reveals a fundamental algebraic structure governing the CNDC. By artificially decoupling the dependencies at the intersecting multi-vertices, we can transform a monolithic, globally coupled conformality matrix into a highly localized system. To formalize this procedure, we first introduce the following matrices.

\begin{definition}\label{def:decoupled_matrices}
Given a CNDC consisting of $s$ T $l$-edges, $n_{\textit{CNDC}}$ vertices and $m$ multi-vertices, we expand the conformality system by introducing independent horizontal and vertical cofactor variables for each multi-vertex, which appends $m$ additional columns to the constraint matrix.
\begin{itemize}
    \item The \textbf{Decoupled Conformality Matrix}, denoted as $K_1$, is an $s(d+1) \times (n_{\textit{CNDC}} + m)$ block-diagonal matrix encoding the smoothness constraints of the $s$ T $l$-edges completely independently. Each of its $s$ blocks operates strictly on the localized variables associated with its corresponding T $l$-edge.
    \item The \textbf{Coupling Matrix}, denoted as $K_2$, is an $m \times 2m$ sparse matrix encoding the consistency conditions. It enforces exactly $m$ equalities, ensuring that the decoupled vertex cofactors at each multi-vertex satisfy the cross-edge equality conditions (i.e., $\delta_i^h = \delta_i^v$ for $i = 1, \dots, m$). When $K_2$ is used together with $K_1$, the omitted zero columns corresponding to mono-vertex cofactors are understood to be restored by zero-padding.
    \item The \textbf{Condensed Coupling Matrix}, denoted as $K \in \mathbb{R}^{m \times (n_{\textit{CNDC}} + m - s(d+1))}$, is defined by: 
    \[ K = K_2 Z, \]
    where $Z \in \mathbb{R}^{2m \times (n_{\textit{CNDC}} + m - s(d+1))}$ is obtained by extracting the $2m$ rows corresponding to the decoupled multi-vertex cofactors from the full basis matrix of $\ker(K_1)$. Different choices of the basis of $\ker(K_1)$ only right-multiply $Z$ by a nonsingular matrix. Hence the matrix $K$ itself may depend on the chosen local bases, but $\mathrm{rank}(K)$ is independent of this choice.
\end{itemize}
\end{definition}

This decoupling mechanism allows us to bypass the conformality matrix $M_1$ entirely. Because the decoupled matrix $K_1$ consists of independent blocks of Vandermonde-like matrices, its row rank is explicitly known. This localized property enables us to formulate a precise and reduced expression for the rank of the conformality matrix corresponding to the CNDC.

\begin{theorem}\label{thm:cndc_rank}
Let $M_1 \in \mathbb{R}^{s(d+1) \times n_{\textit{CNDC}}}$ be the conformality matrix corresponding to the CNDC with $s$ T $l$-edges and $m$ multi-vertices. Then the rank of $M_1$ can be explicitly expressed as:
\begin{equation}
\mathrm{rank}(M_1) = s(d+1) - m + \mathrm{rank}(K),
\end{equation}
where $K$ is the condensed coupling matrix defined in Definition~\ref{def:decoupled_matrices}.
\end{theorem}

\begin{proof}
By the rank-nullity theorem, the dimension of the null space of the conformality matrix $M_1$ corresponding to the CNDC is given by:
\begin{equation}\label{eq:proof_null_M1}
\dim(\ker(M_1)) = n_{\textit{CNDC}} - \mathrm{rank}(M_1).
\end{equation}

Since each of the $m$ multi-vertices is split into two independent horizontal and vertical components, the column dimension of the system changes from $n_{\textit{CNDC}}$ to $n_{\textit{CNDC}} + m$. The decoupled conformality matrix $K_1$ operates on these $n_{\textit{CNDC}} + m$ columns. Since $K_1$ is composed of $s$ independent Vandermonde-like submatrices and every T $l$-edge contains at least $d+2$ vertices by the standing assumption in Section~\ref{sec:prelim}, each block has full row rank $d+1$. Therefore, $\mathrm{rank}(K_1) = s(d+1)$.

By applying the rank-nullity theorem to $K_1$, the dimension of its null space—representing the total number of block-wise free parameters across the entire CNDC—is exactly the number of columns minus the number of rows:
\[ \dim(\ker(K_1)) = (n_{\textit{CNDC}} + m) - s(d+1). \]

With respect to the chosen basis of $\ker(K_1)$, every vector in $\ker(K_1)$ has a unique coefficient vector $\boldsymbol{\alpha} \in \mathbb{R}^{n_{\textit{CNDC}} + m - s(d+1)}$, and the corresponding decoupled multi-vertex cofactors are given by $\boldsymbol{\delta}_{\text{dec}} = Z \boldsymbol{\alpha}$. If another basis of $\ker(K_1)$ is chosen, then $Z$ is replaced by $ZA$ for some nonsingular matrix $A$, and consequently $K$ is replaced by $KA$; therefore $\mathrm{rank}(K)$ is well-defined.

To enforce the original topological connectivity, these decoupled vertex cofactors must simultaneously satisfy the consistency conditions encoded by $K_2$, meaning $K_2 \boldsymbol{\delta}_{\text{dec}} = \mathbf{0}$. Substituting the parameterization yields
\[ K_2 (Z \boldsymbol{\alpha}) = (K_2 Z) \boldsymbol{\alpha} = K \boldsymbol{\alpha} = \mathbf{0}. \]
The map from an original cofactor vector satisfying $M_1\boldsymbol{\delta}=\mathbf{0}$ to the corresponding decoupled vector satisfying $\delta_i^h=\delta_i^v$ is bijective, and the latter vectors are exactly the vectors in $\ker(K_1)$ whose coefficients satisfy $K\boldsymbol{\alpha}=\mathbf{0}$. Hence the solution space of $M_1\boldsymbol{\delta}=\mathbf{0}$ is naturally isomorphic to $\{\boldsymbol{\alpha}:K\boldsymbol{\alpha}=\mathbf{0}\}$. The dimension of this parameter space satisfies
\begin{align}\label{eq:proof_intersection}
\dim(\ker(M_1)) &= \dim(\ker(K)) \nonumber \\
&= (n_{\textit{CNDC}} + m - s(d+1)) - \mathrm{rank}(K).
\end{align}

Equating (\ref{eq:proof_null_M1}) and (\ref{eq:proof_intersection}), we obtain:
\[ n_{\textit{CNDC}} - \mathrm{rank}(M_1) = n_{\textit{CNDC}} + m - s(d+1) - \mathrm{rank}(K). \]
Subtracting $n_{\textit{CNDC}}$ from both sides and rearranging the terms yields:
\[ \mathrm{rank}(M_1) = s(d+1) - m + \mathrm{rank}(K). \]
This completes the proof.
\end{proof}

Combining Theorem~\ref{thm:cndc_dim} and Theorem~\ref{thm:cndc_rank}, we can readily derive the following corollary, which explicitly evaluates the dimension of a spline space using the rank of the condensed coupling matrix $K$:

\begin{corollary}\label{cor:cndc_spline_dim}
Let $\mathscr{T}$ be a T-mesh and $T(\mathscr{T})$ be its T-connected component with a complete partition $\{T_1, T_2\}$, where $T_1$ is the CNDC containing $m$ multi-vertices. Then, the dimension of the spline space $S_d(\mathscr{T})$ can be explicitly evaluated via the condensed coupling matrix $K$ as:
\begin{equation}\label{eq:dim_formula_K}
\dim S_{d}(\mathscr{T}) = (d+1)^2 + (c-t)(d+1) + n_v + m - \mathrm{rank}(K),
\end{equation}
where $c$ is the number of all cross-cut edges, $t$ is the total number of T $l$-edges in $T(\mathscr{T})$, $n_v$ is the number of all interior vertices, and $K$ is the condensed coupling matrix defined in Definition~\ref{def:decoupled_matrices}.
\end{corollary}

\begin{proof}
By Theorem~\ref{thm:cndc_dim}, the dimension of the spline space $S_d(\mathscr{T})$ depends on the rank of the conformality matrix $M_1$ corresponding to the CNDC as follows:
\[
\dim S_{d}(\mathscr{T}) = (d+1)^2 + (c+s-t)(d+1) + n_v - \mathrm{rank}(M_1).
\]
According to Theorem~\ref{thm:cndc_rank}, the rank of $M_1$ can be decoupled and expressed in terms of the condensed coupling matrix $K$ as $\mathrm{rank}(M_1) = s(d+1) - m + \mathrm{rank}(K)$. Substituting this relation directly into the dimension formula yields:
\begin{align*}
\dim S_{d}(\mathscr{T}) &= (d+1)^2 + (c+s-t)(d+1) + n_v - \big(s(d+1) - m + \mathrm{rank}(K)\big) \\
&= (d+1)^2 + c(d+1) + s(d+1) - t(d+1) + n_v - s(d+1) + m - \mathrm{rank}(K).
\end{align*}
Expanding the terms reveals that the components involving the local CNDC edge count $s$ cancel out exactly, i.e., $s(d+1) - s(d+1) = 0$. Combining the remaining terms immediately results in:
\[
\dim S_{d}(\mathscr{T}) = (d+1)^2 + (c-t)(d+1) + n_v + m - \mathrm{rank}(K).
\]
This completes the proof.
\end{proof}

In summary, the decoupling technique offers several profound theoretical and computational advantages over the traditional monolithic evaluation of the conformality system. By replacing the shared multi-vertex cofactors with independent horizontal and vertical variable pairs, this technique prevents the horizontal and vertical edge subsystems from overlapping, thereby partitioning the global matrix into independent diagonal blocks.

Consequently, the computation of the conformality matrix rank is successfully parallelized down to the level of individual T $l$-edges, transforming a massive, coupled global problem into independent local subproblems. Furthermore, because each decoupled block inherits a structured Vandermonde-like form, its local null space can be derived completely analytically using the Lagrange basis polynomials provided in Lemma~\ref{lemma:vander_rref}. This entirely bypasses the need for expensive, case-by-case elementary row operations on complex mesh topologies. Ultimately, the global consistency requirements are compressed into the highly reduced coupling matrix $K$, whose row dimension is $m$ and whose column dimension $n_{\textit{CNDC}}+m-s(d+1)$ is the total number of block-wise free parameters, providing a uniform, scalable, and parallelizable framework for spline dimension evaluation.

The complete step-by-step procedure of this proposed decoupling framework is formalized in Algorithm~\ref{alg:decoupling_method}.

\begin{algorithm}[htbp]
\caption{Decoupling Framework for Calculating Spline Space Dimension}
\label{alg:decoupling_method}
\begin{algorithmic}[1]
\REQUIRE A T-mesh $\mathscr{T}$ and a degree $d$.
\ENSURE The exact dimension of the spline space $S_d(\mathscr{T})$.

\STATE \textbf{CNDC Extraction:} Extract the unique CNDC $T_1$ from $\mathscr{T}$ using Algorithm 1 in \cite{huang2024stability}, thereby obtaining the number of T $l$-edges $s$, vertices $n_{\textit{CNDC}}$, and multi-vertices $m$.
\STATE \textbf{Multi-vertex Cofactor Decoupling:} For each interior multi-vertex $v_i = l_{j_i} \cap l_{k_i}$ ($i=1,\dots,m$), decouple its shared cofactor by introducing a pair of independent horizontal and vertical auxiliary variables $(\delta_i^h, \delta_i^v) \in \mathbb{R}^2$ associated with the T $l$-edges $l_{j_i}$ and $l_{k_i}$ respectively.
\STATE \textbf{Parallel Local Solution:} Evaluate in parallel the constraints along each individual T $l$-edge. Apply Lemma~\ref{lemma:vander_rref} locally on each independent edge block to analytically express the decoupled vertex cofactors via local Lagrange basis functions in terms of the block-wise free parameter vector $\boldsymbol{\alpha}$.
\STATE \textbf{Condensed Coupling Matrix Construction:} Impose the identity constraints $\delta_i^h = \delta_i^v$ to force the decoupled pairs back to the original unique multi-vertex variables. By substituting the local basis expressions from Step 3 into these constraints (represented by the matrix $K_2 \in \mathbb{R}^{m \times 2m}$), directly construct the condensed coupling matrix $K$.
\STATE \textbf{Dimension Calculation:} Determine $\mathrm{rank}(K)$ from the highly reduced condensed matrix, and calculate $\dim S_d(\mathscr{T})$ exactly via the unified dimension formula \eqref{eq:dim_formula_K}.
\end{algorithmic}
\end{algorithm}

\section{Dimension bounds via the decoupling framework}

By leveraging the decoupling framework and the dimension formula established in Section 3, we derive explicit upper and lower bounds for the dimension of the bivariate spline space $S_d(\mathscr{T})$. We first establish a fundamental algebraic link between the condensed coupling matrix $K$ and the reduced multi-vertex conformality matrix $M_1^{\text{mult}}$. Utilizing the relation alongside the dimension formula in Corollary~\ref{cor:cndc_spline_dim}, we formalize a new pair of dimension bounds. The lower bound is algebraically identical to the one in Corollary~\ref{cor:dim_bounds}, whereas the upper bound is obtained from a structural lower estimate for $\mathrm{rank}(K)$. We then analyze the intrinsic algebraic structure of $K$, derive the permutation formula for the upper bound, and compare it with the homological upper bound in~\cite{dim2014}. The comparison shows that, for a fixed T-mesh in the same case of bi-degree $(d,d)$, and smoothness order $(d-1,d-1)$,  the algebraic correction term agrees with Mourrain's ordered-segment term after the notation is translated. The advantage of the present formulation is that the diagonalizable component is separated first, so the ordering problem is reduced from all T $l$-edges to only those in the CNDC. Finally, the example introduced in Section~3 shows that both the upper and lower bounds  can be attained.

\subsection{Dimension bounds from the condensed coupling matrix}\label{subsec:bounds_condensed_matrix}

\begin{lemma}\label{lem:rank_identity}
Let $K \in \mathbb{R}^{m \times \left(n_{\textit{CNDC}}+m-s(d+1)\right)}$ be the condensed coupling matrix defined in Definition~\ref{def:decoupled_matrices}, and let $M_1^{\text{mult}} \in \mathbb{R}^{(s(d+1) - n_{\textit{CNDC}} + m) \times m}$ be the reduced multi-vertex conformality matrix defined in Theorem~\ref{thm:cndc_dim_mult}. Then the ranks of these two matrices satisfy the following linear identity:
\begin{equation}\label{eq:rank_identity_formula}
\mathrm{rank}(K) = \mathrm{rank}(M_1^{\text{mult}}) + n_{\textit{CNDC}} - s(d+1).
\end{equation}
\end{lemma}

\begin{proof}
According to Corollary~\ref{cor:cndc_spline_dim}, the dimension of the spline space $S_d(\mathscr{T})$ expressed via $\mathrm{rank}(K)$ is:
\[
\dim S_{d}(\mathscr{T}) = (d+1)^2 + (c-t)(d+1) + n_v + m - \mathrm{rank}(K).
\]
Meanwhile, Theorem~\ref{thm:cndc_dim_mult} provides the dimension formula in terms of $\mathrm{rank}(M_1^{\text{mult}})$:
\[
\dim S_{d}(\mathscr{T}) = (d+1)^2 + (c+s-t)(d+1) + n_v - n_{\textit{CNDC}} + m - \mathrm{rank}(M_1^{\text{mult}}).
\]
Since both formulas uniquely determine the exact dimension of the identical spline space $S_d(\mathscr{T})$, we can equate the two expressions directly:
\begin{align*}
(d+1)^2 + (c-t)(d+1) + n_v + m - \mathrm{rank}(K) = & \, (d+1)^2 + (c+s-t)(d+1) \\
& + n_v - n_{\textit{CNDC}} + m - \mathrm{rank}(M_1^{\text{mult}}).
\end{align*}
By canceling the common terms $(d+1)^2$, $(c-t)(d+1)$, $n_v$, and $m$ from both sides, the equation simplifies to:
\[
\mathrm{rank}(K) = \mathrm{rank}(M_1^{\text{mult}}) + n_{\textit{CNDC}} - s(d+1).
\]
This completes the proof.
\end{proof}

The rank identity established in Lemma~\ref{lem:rank_identity} indicates that the reduced multi-vertex conformality matrix $M_1^{\text{mult}}$ and the condensed coupling matrix $K$ are \textbf{not the same matrix}. Nevertheless, according to Corollary~\ref{cor:cndc_spline_dim} and Theorem~\ref{thm:cndc_dim_mult}, both $M_1^{\text{mult}}$ and $K$ simultaneously encode the stability information of the spline space dimension. Despite their shared ability to characterize dimension stability, their mathematical formulations and computational behaviors are fundamentally different. The matrix $M_1^{\text{mult}}$ is obtained through sequential elementary row operations that progressively eliminate single-vertex constraints. This global elimination process couples cross-edge degrees of freedom, making the formulation structurally complex, computationally expensive, and difficult to compute in parallel. In contrast, the condensed matrix $K$ is constructed through a decoupling approach that utilizes independent local edge bases. By isolating the local geometric constraints into distinct blocks, $K$ completely avoids costly global row operations and substantially reduces the system size to depend only on the number of multi-vertices $m$. This decoupled structure not only simplifies the algebraic formulation but also offers a high degree of block-wise parallelism, making $K$ much more efficient and scalable than $M_1^{\text{mult}}$ for practical computation.

By leveraging the algebraic structure of the condensed coupling matrix $K$ and estimating its maximum possible rank, we can establish a sharp, geometrically intuitive lower bound for the global spline space dimension without explicitly calculating the matrix rank.

\begin{theorem}\label{thm:dim_lower_bound_K}
Let $K \in \mathbb{R}^{m \times \dim(\ker K_1)}$ be the condensed coupling matrix defined in Definition~\ref{def:decoupled_matrices}. The dimension of the spline space $S_d(\mathscr{T})$ is bounded from below by:
\begin{equation}\label{eq:lower_bound_via_K}
\dim S_{d}(\mathscr{T}) \ge L_K(\mathscr{T}),
\end{equation}
where the lower bound $L_K(\mathscr{T})$ is given by:
\[
L_K(\mathscr{T}) = (d+1)^2 + (c-t)(d+1) + n_v + \big(s(d+1) - n_{\textit{CNDC}}\big)_+,
\]
and $(\cdot)_+ = \max(\cdot, 0)$ denotes the non-negative truncation operator.
\end{theorem}

\begin{proof}
Since $K$ has $m$ rows and $n_{\textit{CNDC}}+m-s(d+1)$ columns,
\[
\mathrm{rank}(K) \le \min\big(m, \, n_{\textit{CNDC}} + m - s(d+1)\big).
\]
Subtracting this inequality from the exact dimension formula in Corollary~\ref{cor:cndc_spline_dim} gives
\[
\dim S_{d}(\mathscr{T}) \ge (d+1)^2 + (c-t)(d+1) + n_v + m - \min\big(m, \, n_{\textit{CNDC}} + m - s(d+1)\big).
\]
By employing the algebraic identity $m - \min(m, X) = (m - X)_+$ and substituting $X = n_{\textit{CNDC}} + m - s(d+1)$, the terminal block simplifies as follows:
\[
m - \min\big(m, \, n_{\textit{CNDC}} + m - s(d+1)\big) = \big(m - (n_{\textit{CNDC}} + m - s(d+1))\big)_+ = \big(s(d+1) - n_{\textit{CNDC}}\big)_+.
\]
Plugging this result back into the inequality directly confirms the lower bound $L_K(\mathscr{T})$. This completes the proof.
\end{proof}

\begin{proposition}\label{prop:lower_bound_equivalence}
The lower bound $L_K(\mathscr{T})$ defined in Theorem~\ref{thm:dim_lower_bound_K} is algebraically identical to the traditional lower bound $L(\mathscr{T})$ derived via $M_1^{\text{mult}}$ in Corollary~\ref{cor:dim_bounds}.
\end{proposition}

\begin{proof}
Recall the traditional lower bound formula from Corollary~\ref{cor:dim_bounds}:
\begin{equation}\label{eq:traditional_lower_bound}
L(\mathscr{T}) = (d+1)^2 + (c+s-t)(d+1) + n_v - n_{\textit{CNDC}} + \big(n_{\textit{CNDC}} - s(d+1)\big)_+.
\end{equation}
We establish the identity $L_K(\mathscr{T}) \equiv L(\mathscr{T})$ by evaluating both expressions under two cases.

\medskip
\noindent\textbf{Case 1:} $s(d+1) \ge n_{\textit{CNDC}}$ \\
In this regime, the truncation terms evaluate to $\big(n_{\textit{CNDC}} - s(d+1)\big)_+ = 0$ and $\big(s(d+1) - n_{\textit{CNDC}}\big)_+ = s(d+1) - n_{\textit{CNDC}}$. Both formulations collapse directly to:
\[
L(\mathscr{T}) = L_K(\mathscr{T}) = (d+1)^2 + (c-t)(d+1) + n_v + s(d+1) - n_{\textit{CNDC}}.
\]

\medskip
\noindent\textbf{Case 2:} $s(d+1) < n_{\textit{CNDC}}$\\
In this regime, the truncation terms evaluate to $\big(n_{\textit{CNDC}} - s(d+1)\big)_+ = n_{\textit{CNDC}} - s(d+1)$ and $\big(s(d+1) - n_{\textit{CNDC}}\big)_+ = 0$. Substituting these terms into \eqref{eq:traditional_lower_bound} and $L_K(\mathscr{T})$ yields direct algebraic cancellation of $n_{\textit{CNDC}}$ and $s(d+1)$ respectively, reducing both bounds to:
\[
L(\mathscr{T}) = L_K(\mathscr{T}) = (d+1)^2 + (c-t)(d+1) + n_v.
\]

\medskip
Since $L_K(\mathscr{T}) = L(\mathscr{T})$ holds in both scenarios, the proof is complete.
\end{proof}

It should be noted that the lower bound established in Theorem~\ref{thm:dim_lower_bound_K} is an estimate derived within the T-connected component framework, ensuring that the bound obtained from $M_1^{\text{mult}}$ remains consistent with that from $K$. While this lower bound is already effective, a corresponding upper bound for the space dimension has not yet been established. Therefore, in the following, we will provide a sharp upper bound for the spline space dimension.

Before introducing the main theorem, we first illustrate the structure of the condensed coupling matrix $K$ through a concrete example.

\begin{example}\label{ex:matrix_structure_illustration}
Consider the condensed coupling matrix $K$ from Example~\ref{ex:decoupled_conformality_matrix}. The columns of $K$ can be divided into $4$ blocks, where each block corresponds to the number of degrees of freedom of the decoupled vertex cofactors on a T $l$-edge (which is exactly $1$ degree of freedom per T $l$-edge). Meanwhile, the rows correspond to the coupling conditions at the multi-vertices. To explicitly visualize this structure, we label the columns with the free parameters $\alpha_j$ and the rows with the multi-vertices as follows:
\[
K = \bordermatrix{
    & \alpha_1 & \alpha_3 & \alpha_2 & \alpha_4 \cr
    v_{11} & -g_5^{(1)}(t_4) & 0 & g_4^{(2)}(s_5) & 0 \cr
    v_{12} & 1 & 0 & 0 & -1 \cr
    v_{13} & 0 & -g_5^{(3)}(t_4) & -1 & 0 \cr
    v_{14} & 0 & 1 & 0 & g_5^{(4)}(s_4)
}.
\]
Since for each T $l$-edge, the number of multi-vertices is greater than the local degree of freedom 1, there must be one multi-vertex within each row block acting as a free variable. For instance, for the parameter $\alpha_1$, the horizontal multi-vertex cofactor $\delta_{12}^h$ at $v_{12}$ serves as the free variable satisfying $\delta_{12}^h = \alpha_1$, which directly yields the entry $1$ in the second row and first column of $K$.
\[
\begin{pmatrix}
-g_5^{(1)}(t_4) & 0 & g_4^{(2)}(s_5) & 0 \\
\boxed1 & 0 & 0 & -1 \\
0 & -g_5^{(3)}(t_4) & -1 & 0 \\
0 & \boxed1 & 0 & g_5^{(4)}(s_4)
\end{pmatrix}
\]

Next, we consider identifying a maximal-order square upper triangular submatrix in the following manner: starting from the first column, we look for a row containing a $1$ or $-1$ such that all other non-zero elements in that row are located strictly to its right. We then proceed sequentially column by column through the following steps:
\begin{itemize}
    \item \textbf{Step 1:} For the first column, the second row contains a $1$, and its remaining non-zero element ($-1$) lies strictly to its right.
    \item \textbf{Step 2:} For the second column, the fourth row contains a $1$, and its other non-zero element ($g_5^{(4)}(s_4)$) lies strictly to its right. Beyond these, no further eligible rows can be found for the remaining columns, and the search terminates.
\end{itemize}

This extraction process targets the identified second and fourth rows, yielding a $2 \times 4$ submatrix from which a $2 \times 2$ square upper triangular submatrix with a rank of $2$ is isolated:
\[
\begin{pmatrix}
-g_5^{(1)}(t_4) & 0 & g_4^{(2)}(s_5) & 0 \\
1 & 0 & 0 & -1 \\
0 & -g_5^{(3)}(t_4) & -1 & 0 \\
0 & 1 & 0 & g_5^{(4)}(s_4)
\end{pmatrix}
\longrightarrow
\begin{pmatrix}
1 & 0 & 0 & -1 \\
0 & 1 & 0 & g_5^{(4)}(s_4)
\end{pmatrix}
\longrightarrow
\begin{pmatrix}
1 & 0 \\
0 & 1
\end{pmatrix}
\]

By considering all possible permutations of the T $l$-edges, we find that there exists an arrangement that yields a $3 \times 3$ square upper triangular submatrix. Specifically, this is achieved under the ordering of the free parameters $(\alpha_1, \alpha_2, \alpha_3, \alpha_4)$, which corresponds to the ordering of the T $l$-edges represented as $v_1v_7 \succ v_3v_4 \succ v_2v_8 \succ v_5v_6$. This column permutation reorders the matrix, allowing rows 2, 3, and 4 to be sequentially selected to form a $3 \times 3$ square upper triangular submatrix, demonstrating that the rank of $K$ is at least $3$:
\[
\bordermatrix{
    & \alpha_1 & \alpha_2 & \alpha_3 & \alpha_4 \cr
    & -g_5^{(1)}(t_4) & g_4^{(2)}(s_5) & 0 & 0 \cr
    & 1 & 0 & 0 & -1 \cr
    & 0 & -1 & -g_5^{(3)}(t_4) & 0 \cr
    & 0 & 0 & 1 & g_5^{(4)}(s_4)
}
\longrightarrow
\begin{pmatrix}
1 & 0 & 0 \\
0 & -1 & -g_5^{(3)}(t_4) \\
0 & 0 & 1
\end{pmatrix}
\]
Furthermore, Example~\ref{ex:decoupled_conformality_matrix} shows that the estimate $\mathrm{rank}(K)\geq 3$ is attained: when the determinant condition in that example vanishes, one has $\mathrm{rank}(K)=3$. Thus the rank estimate obtained here cannot be increased for this T-mesh without using additional geometric information.
\end{example}

Based on the row-selection mechanisms illustrated in Example~\ref{ex:matrix_structure_illustration}, we now establish a general estimate for the rank of the condensed coupling matrix $K$ by identifying its maximal-order upper triangular submatrix.

\begin{theorem}\label{thm:combinatorial_rank_K}
Let $\{l_1, l_2, \dots, l_s\}$ be the set of all $s$ T $l$-edges of the CNDC in the T-mesh $\mathscr{T}$. Let $\pi$ be a permutation of these $s$ T $l$-edges that prescribes their ordering. For each T $l$-edge $l_i$ ($i=1, 2, \dots, s$), let $u_\pi^+(l_i)$ denote its forward multi-vertex intersection count under the permutation $\pi$, defined as the number of multi-vertices on $l_i$ shared with the T $l$-edges that succeed $l_i$ (denoted as $l_i \succ l_j$). 

Then, the rank of the condensed coupling matrix $K$ satisfies the following combinatorial lower bound:
\begin{equation}\label{eq:max_triangular_minor}
    \mathrm{rank}(K) \ge \tau_{\max} = \max_{\pi} \sum_{i=1}^s \min\big\{n(l_i) - d - 1, u_\pi^+(l_i)\big\},
\end{equation}
where $n(l_i)$ is the number of vertices on $l_i$, and $d$ is the given spline degree.
\end{theorem}

\begin{proof}
We prove the estimate for an arbitrary fixed permutation $\pi$ and then maximize over all permutations. Write
\[
q(l_i)=n(l_i)-d-1
\]
for the local nullity of the Vandermonde system on $l_i$.

\begin{itemize}
    \item[Step1:]\textbf{Local choice of free vertices.} Let $l$ be a T $l$-edge with $n(l)$ vertices and local nullity $q(l)$. For any subset $B$ of vertices on $l$ with $|B|\le q(l)$, one can choose a local null-space basis so that the rows indexed by $B$ contain an identity block of order $|B|$. Indeed, extend $B$ to a set of $q(l)$ free vertices, the remaining $d+1$ vertices may be used as reference vertices, and the corresponding Vandermonde matrix is nonsingular because the nodes on the T $l$-edge are distinct. Solving the reference variables in terms of the free variables gives a local basis in which the chosen free rows are exactly the coordinate rows. This is the only point where the Vandermonde structure is used.
    \item[Step2:]\textbf{Selecting admissible pivot rows.} For a fixed T $l$-edge $l_i$, let $U_\pi^+(l_i)$ be the set of multi-vertices on $l_i$ that are shared with succeeding T $l$-edges in the order $\pi$. Set
\[
k_i=\min\{q(l_i),u_\pi^+(l_i)\}.
\]
Choose a subset $B_i\subseteq U_\pi^+(l_i)$ with $|B_i|=k_i$. By Step 1, the local basis on $l_i$ can be chosen so that the rows corresponding to $B_i$ give $k_i$ pivot rows in the column block associated with $l_i$. These choices are made independently for different T $l$-edges because the decoupled local systems are independent.
   \item[Step3:]\textbf{Triangular structure.} Consider a row corresponding to a multi-vertex in $B_i$. This multi-vertex lies on $l_i$ and on a succeeding T $l$-edge, say $l_j$ with $l_i\succ l_j$. Therefore, in the coupling row, the only possible nonzero entries occur in the column blocks belonging to $l_i$ and $l_j$. There is no nonzero entry in any column block preceding $l_i$. Hence, after the selected rows and pivot columns are arranged according to the order $\pi$, all entries below the selected diagonal block positions vanish. The selected submatrix is therefore block upper triangular, and its diagonal blocks are identity matrices of sizes $k_i$.
\end{itemize}

The selected row sets $B_i$ are disjoint. Indeed, each multi-vertex is the intersection of exactly two T $l$-edges, and it is counted as a forward intersection only for the earlier one of the two edges in the ordering. Thus the construction produces a nonsingular square submatrix of $K$ of order
\[
\sum_{i=1}^s k_i=
\sum_{i=1}^s \min\{n(l_i)-d-1, u_\pi^+(l_i)\}.
\]
Consequently, for this fixed permutation $\pi$, the rank of $K$ is at least the above number. Taking the maximum over all permutations gives \eqref{eq:max_triangular_minor}.
\end{proof}

Alternatively, the maximal order $\tau_{\max}$ can be expressed through backward intersections. This form is useful because it converts the rank certificate into the correction term appearing in the upper dimension bound.

\begin{corollary}\label{cor:defect_formulation}
Let $m$ be the number of multi-vertices in the CNDC of T-mesh $\mathscr{T}$. For a given permutation $\pi$ of the $s$ T $l$-edges, let $u_\pi^-(l_i)$ denote the backward multi-vertex intersection count of $l_i$, defined as the number of multi-vertices on $l_i$ shared with the T $l$-edges that precede $l_i$ (denoted as $l_j \succ l_i$). 
Then, the maximal order $\tau_{\max}$ defined in Theorem~\ref{thm:combinatorial_rank_K} satisfies the following equivalent backward-intersection formula:
\begin{equation}\label{eq:max_triangular_defect}
    \tau_{\max} = m - \min_{\pi} \sum_{i=1}^s \big[ m(l_i) - n(l_i) + d + 1 - u_\pi^-(l_i) \big]_+,
\end{equation}
where $m(l_i)$ represents the number of multi-vertices located on $l_i$, and $[\cdot]_+ = \max\{\cdot, 0\}$.
\end{corollary}

\begin{proof}
For any T $l$-edge $l_i$, the number of multi-vertices $m(l_i)$ is uniquely partitioned into backward and forward intersections under the permutation $\pi$. This yields the local identity $m(l_i) = u_\pi^-(l_i) + u_\pi^+(l_i)$. 

The summand in the lower bound \eqref{eq:max_triangular_minor} can be rewritten by substituting $u_\pi^+(l_i) = m(l_i) - u_\pi^-(l_i)$:
\begin{align*}
    \min\big\{n(l_i) - d - 1, u_\pi^+(l_i)\big\} &= u_\pi^+(l_i) - \big[ u_\pi^+(l_i) - (n(l_i) - d - 1) \big]_+ \\
    &= u_\pi^+(l_i) - \big[ m(l_i) - u_\pi^-(l_i) - n(l_i) + d + 1 \big]_+.
\end{align*}

Summing the rewritten expression over all $s$ blocks yields:
\[
    \sum_{i=1}^s \min\big\{n(l_i) - d - 1, u_\pi^+(l_i)\big\} = \sum_{i=1}^s u_\pi^+(l_i) - \sum_{i=1}^s \big[ m(l_i) - n(l_i) + d + 1 - u_\pi^-(l_i) \big]_+.
\]
Since each multi-vertex is formed by the intersection of exactly two perpendicular T $l$-edges, it is classified as a forward intersection exactly once---specifically for the T $l$-edge that appears earlier in the sequence prescribed by $\pi$. This implies $\sum_{i=1}^s u_\pi^+(l_i) = m$. Substituting this relation into the summation gives:
\[
    \sum_{i=1}^s \min\big\{n(l_i) - d - 1, u_\pi^+(l_i)\big\} = m - \sum_{i=1}^s \big[ m(l_i) - n(l_i) + d + 1 - u_\pi^-(l_i) \big]_+.
\]
Finally, maximizing the left-hand side over all possible choices of $\pi$ is mathematically equivalent to minimizing the subtracted correction sum on the right-hand side, which yields \eqref{eq:max_triangular_defect} and completes the proof.
\end{proof}

By restricting the permutations to a specific directional layout—namely, ordering all horizontal (vertical) T $l$-edges before all vertical (horizontal) ones—we can derive a simplified lower bound for $\tau_{\max}$.

\begin{remark}\label{rem:tau_dp}
The direct evaluation of $\tau_{\max}$ requires $s!$ permutations of all the T l-edges of CNDC, but it can be computed much more efficiently in a  recursive fashion. In fact, for each subset $E\subseteq\{l_1,\ldots,l_s\}$ of T l-edges, let $F(E)$ denote the maximum of the sum in \eqref{eq:max_triangular_minor} over all permutations of the T $l$-edges in $E$, and set $F(\varnothing)=0$. Then
\[
F(E)=\max_{l\in E}\left\{F(E\setminus\{l\})+\min\left(n(l)-d-1,\ \#\bigl\{l'\in E\setminus\{l\}: l\cap l'\text{ is a multi-vertex}\bigr\}\right)\right\}.
\]
Indeed, if $l$ is put in the first place in a permutation of $E$, every T $l$-edge in $E\setminus\{l\}$ that intersects $l$ succeeds it, so the second term in the above formula  is exactly the contribution of $l$; the remaining contribution is $F(E\setminus\{l\})$. Therefore,
\[
\tau_{\max}=F(\{l_1,\ldots,l_s\}).
\]
To avoid repeatedly checking the intersections of all the T l-edges, we first record the incidence relation between T $l$-edges: for each T $l$-edge, we store the set of T $l$-edges that intersect it at a multi-vertex. Then the number
\[
\#\bigl\{l'\in E\setminus\{l\}: l\cap l'\text{ is a multi-vertex}\bigr\}
\]
can be evaluated as follows. For each T $l$-edge $l$, define its CNDC incidence set
\[
\mathcal{A}(l)=\bigl\{l'\ne l: l\cap l'\text{ is a multi-vertex}\bigr\}.
\]
For a current subset $E$, the required number is simply $|\mathcal{A}(l)\cap E|$. Thus the recursion uses only the fixed incidence relation of the CNDC; no geometric coordinates or repeated reconstruction of intersections is required. Since there are $2^s$ subsets and at most $s$ choices of the first edge for each subset, the recursion has $O(s2^s)$ transitions once these intersection counts are recorded together with the subsets. A completely direct implementation, which recounts the relevant intersections each time, gives $O(s^2 2^s)$ elementary checks which is far smaller than enumerating $s!$ permutations. The storage needed for the values $F(E)$ is $O(2^s)$. Corollary~\ref{cor:directional_bound} below presents a simpler estimate when an exact computation of $\tau_{\max}$ is unnecessary.
\end{remark}

\begin{corollary}\label{cor:directional_bound}

Let the set of $s$ T $l$-edges be partitioned into the horizontal edge set $H_1$ and the vertical edge set $H_2$. Then,
\begin{equation}\label{eq:directional_lower_bound_general}
\tau_{\max}\ge
\max\left\{
\sum_{l\in H_1}\min\{n(l)-d-1,m(l)\},
\sum_{l\in H_2}\min\{n(l)-d-1,m(l)\}
\right\}.
\end{equation}
Moreover, for the CNDC obtained from a complete partition, each T $l$-edge satisfies $n(l)-d-1\le m(l)$. Hence, if $H_{\ell}\in\{H_1,H_2\}$ is chosen so that
\[
\sum_{l\in H_{\ell}}\bigl(n(l)-d-1\bigr)
=
\max\left\{
\sum_{l\in H_1}\bigl(n(l)-d-1\bigr),
\sum_{l\in H_2}\bigl(n(l)-d-1\bigr)
\right\},
\]
then
\begin{equation}\label{eq:directional_lower_bound}
    \tau_{\max} \ge \sum_{l \in H_{\ell}} \big( n(l) - d - 1 \big).
\end{equation}

\end{corollary}

\begin{proof}

We firstly order all edges in $H_1$ before all edges in $H_2$. For every $l_i\in H_1$, all its multi-vertices are shared with succeeding vertical edges, so $u_{\pi}^{+}(l_i)=m(l_i)$; for every $l_j\in H_2$, one has $u_{\pi}^{+}(l_j)=0$. Substituting these counts into \eqref{eq:max_triangular_minor} gives
\[
\tau_{\max}\ge \sum_{l\in H_1}\min\{n(l)-d-1,m(l)\}.
\]
The same argument with the vertical edges ordered before the horizontal edges gives the analogous estimate over $H_2$, and \eqref{eq:directional_lower_bound_general} follows by taking the larger of the two values.

For the simplified form, note that if a T $l$-edge $l$ in the CNDC had more than $d$ mono-vertices, then after removing its intersections with the other CNDC edges it would still contain at least $d+1$ remaining vertices. Such an edge could be ordered into the diagonalizable component, contradicting with the completeness of the CNDC. Thus $n(l)-m(l)\le d$, and in particular $n(l)-d-1\le m(l)$. Substituting this into \eqref{eq:directional_lower_bound_general} yields \eqref{eq:directional_lower_bound}.
\end{proof}

By combining the dimension formula in Corollary~\ref{cor:cndc_spline_dim} with the combinatorial lower bound on the rank of the condensed coupling matrix in Theorem~\ref{thm:combinatorial_rank_K}, we can establish an explicit upper bound for the dimension of the spline space $S_d(\mathscr{T})$.

\begin{theorem}\label{thm:spline_dim_upper_bound}
Let $\mathscr{T}$ be a T-mesh satisfying the conditions of Corollary~\ref{cor:cndc_spline_dim}. The dimension of the spline space $S_d(\mathscr{T})$ is bounded from above in the following backward-intersection form:
\begin{equation}\label{eq:spline_dim_upper_bound}
    \dim S_{d}(\mathscr{T}) \le U_{K}(\mathscr{T}),
\end{equation}
where $c$ is the number of all cross-cut edges, $t$ is the number of T $l$-edges in $T(\mathscr{T})$, $n_v$ is the number of all interior vertices and the upper bound $U_K(\mathscr{T})$ is given by:
$$U_K(\mathscr{T})=(d+1)^2 + (c-t)(d+1) + n_v + \min_{\pi} \sum_{i=1}^s \big[ m(l_i) - n(l_i) + d + 1 - u_\pi^-(l_i) \big]_+.$$
\end{theorem}

\begin{proof}
From Corollary~\ref{cor:cndc_spline_dim}, the exact dimension of the spline space $S_d(\mathscr{T})$ is given by:
\[
    \dim S_{d}(\mathscr{T}) = (d+1)^2 + (c-t)(d+1) + n_v + m - \mathrm{rank}(K).
\]
Theorem~\ref{thm:combinatorial_rank_K} provides a lower bound on the rank of the condensed coupling matrix, stating that $\mathrm{rank}(K) \ge \tau_{\max}$. Substituting this relation into the exact dimension formula produces the initial upper bound:
\[
    \dim S_{d}(\mathscr{T}) \le (d+1)^2 + (c-t)(d+1) + n_v + m - \tau_{\max}.
\]
According to the equivalent backward-intersection formula established in Corollary~\ref{cor:defect_formulation}, the maximal order $\tau_{\max}$ can be written as:
\[
    \tau_{\max} = m - \min_{\pi} \sum_{i=1}^s \big[ m(l_i) - n(l_i) + d + 1 - u_\pi^-(l_i) \big]_+.
\]
By substituting this expression into the initial upper bound, the term representing the number of multi-vertices $m$ cancels out directly:
\begin{align*}
    \dim S_{d}(\mathscr{T}) &\le (d+1)^2 + (c-t)(d+1) + n_v + m - \left( m - \min_{\pi} \sum_{i=1}^s \big[ m(l_i) - n(l_i) + d + 1 - u_\pi^-(l_i) \big]_+ \right) \\
    &= (d+1)^2 + (c-t)(d+1) + n_v + \min_{\pi} \sum_{i=1}^s \big[ m(l_i) - n(l_i) + d + 1 - u_\pi^-(l_i) \big]_+.
\end{align*}
This simplifies exactly to \eqref{eq:spline_dim_upper_bound} and completes the proof.
\end{proof}

The above dimension bounds are sharp in the sense of attainability, rather than to assert equality for every prescribed T-mesh. For the dimensionally unstable T-mesh in Examples~\ref{ex:decoupled_conformality_matrix} and~\ref{ex:matrix_structure_illustration}, with $d=3$, one has $s=4$, $n_{\textit{CNDC}}=16$, $m=4$, and $\tau_{\max}=3$. Therefore,
\[
L_K(\mathscr{T})=(d+1)^2+(c-t)(d+1)+n_v,
\qquad
U_K(\mathscr{T})=L_K(\mathscr{T})+1.
\]
When $\det(K)\neq 0$, Example~\ref{ex:decoupled_conformality_matrix} gives $\mathrm{rank}(K)=4$ and hence $\dim S_d(\mathscr{T})=L_K(\mathscr{T})$; when $\det(K)=0$, it gives $\mathrm{rank}(K)=3$ and hence $\dim S_d(\mathscr{T})=U_K(\mathscr{T})$. Both cases occur for admissible grid coordinates. For instance, take
\[
(t_2,t_3,t_4,t_5,t_6,t_7)=(0,1,2,3,4,5),
\qquad
(s_2,s_3,s_4,s_5,s_6,s_7)=(0,1,2,3,4,5).
\]
Then
\[
g_5^{(1)}(t_4)=\frac32,
\quad g_5^{(3)}(t_4)=\frac23,
\quad g_4^{(2)}(s_5)=\frac32,
\quad g_5^{(4)}(s_4)=\frac32,
\]
so $\det(K)=\frac32-\frac32\cdot\frac23\cdot\frac32=0$ and the upper bound is attained. If only $s_7$ is changed from $5$ to $6$, then $g_5^{(4)}(s_4)=\frac43$ and
\[
\det(K)=\frac32-\frac32\cdot\frac23\cdot\frac43=\frac16\neq 0,
\]
so the lower bound is attained. Thus different geometric realizations of the same unstable T-mesh topology attain the lower and upper bounds, respectively. This is the precise sense in which the bounds derived in this paper are sharp.

By restricting the permutation to the directional sequencing strategy outlined in Corollary~\ref{cor:directional_bound}, the optimization over all permutations can be bypassed, yielding a more explicit and readily computable upper bound.

\begin{corollary}\label{cor:spline_dim_directional_upper_bound}
Under the same assumptions as in Theorem~\ref{thm:spline_dim_upper_bound} and in Corollary~\ref{cor:directional_bound}, the dimension of the spline space $S_d(\mathscr{T})$ satisfies the following directional upper bound:
\begin{equation}\label{eq:spline_dim_directional_upper_bound}
    \dim S_{d}(\mathscr{T}) \le (d+1)^2 + (c-t)(d+1) + n_v + m - \sum_{l \in H_{\ell}} \big( n(l) - d - 1 \big).
\end{equation}
\end{corollary}

\begin{proof}
From the proof of Theorem~\ref{thm:spline_dim_upper_bound}, the exact dimension formula combined with the rank bound $\mathrm{rank}(K) \ge \tau_{\max}$ yields the initial upper bound:
\[
    \dim S_{d}(\mathscr{T}) \le (d+1)^2 + (c-t)(d+1) + n_v + m - \tau_{\max}.
\]
Applying the directional lower bound from Corollary~\ref{cor:directional_bound}, we have $\tau_{\max} \ge \sum_{l \in H_{\ell}} (n(l) - d - 1)$. Multiplying this inequality by $-1$ reverses the inequality sign, giving $-\tau_{\max} \le -\sum_{l \in H_{\ell}} (n(l) - d - 1)$. Substituting this directly into the initial upper bound yields \eqref{eq:spline_dim_directional_upper_bound}, which completes the proof.
\end{proof}

\subsection{Relationship with Mourrain's homological formulation}\label{subsec:mourrain_comparison}

Mourrain's dimension formulas of polynomial spline spaces over T-meshes are written in terms of cells, interior edges, interior vertices, and a homological correction~\cite{dim2014}, whereas the formulas above are written in terms of cross-cuts, T $l$-edges, and conformality-matrix ranks. In this section, we prove relationship between these two formulas. We only focus on the case of the highest order of smoothness. We first prove that these two formulas have the same topological part.

\begin{lemma}\label{lem:mourrain_topological_translation}
Let $f_2$, $f_1^h$, $f_1^v$, and $f_0$ denote, respectively, the numbers of cells, horizontal interior edges, vertical interior edges, and interior vertices in Mourrain's notation, and put $f_1^o=f_1^h+f_1^v$. For bi-degree $(d,d)$ and smoothness order $(d-1,d-1)$, the combinatorial part of Mourrain's dimension formula is
\begin{equation}\label{eq:mourrain_combinatorial_part}
B_{\mathrm M}(\mathscr T)
=(d+1)^2f_2-d(d+1)f_1^o+d^2f_0.
\end{equation}
It satisfies
\begin{equation}\label{eq:mourrain_base_equals_cofactor_base}
B_{\mathrm M}(\mathscr T)
=(d+1)^2+(c-t)(d+1)+n_v.
\end{equation}
\end{lemma}

\begin{proof}
Since the domain is simply connected, its relative Euler relation is
\begin{equation}\label{eq:relative_euler_identity}
f_2-f_1^o+f_0=1.
\end{equation}
We also have
\begin{equation}\label{eq:interior_edge_large_edge_identity}
f_1^o=2f_0+c-t.
\end{equation}
Indeed, let $n(l)$ be the number of interior vertices on an interior $l$-edge $l$. Every interior vertex belongs to exactly one horizontal and one vertical interior $l$-edge, and hence
\[
\sum_{l}n(l)=2f_0.
\]
An interior $l$-edge with $n(l)$ interior vertices consists of $n(l)+1$ interior edges if it is a cross-cut, of $n(l)$ interior edges if it is a ray, and of $n(l)-1$ interior edges if it is a T $l$-edge. Summing over all interior $l$-edges gives \eqref{eq:interior_edge_large_edge_identity}; the ray contributions cancel and therefore do not appear explicitly.

Using \eqref{eq:relative_euler_identity} to write $f_2=1+f_1^o-f_0$ in \eqref{eq:mourrain_combinatorial_part}, we obtain
\begin{align*}
B_{\mathrm M}(\mathscr T)
&=(d+1)^2+(d+1)f_1^o-(2d+1)f_0\\
&=(d+1)^2+(d+1)(2f_0+c-t)-(2d+1)f_0\\
&=(d+1)^2+(c-t)(d+1)+f_0.
\end{align*}
$f_0=n_v$ gives \eqref{eq:mourrain_base_equals_cofactor_base}.
\end{proof}

Next we use 
\[
h_d(\mathscr T)
:=h_{d,d}^{(d-1,d-1)}(\mathscr T)
\]
to denote Mourrain's homological correction term. The folllowing proposition identifies this correction term with the rank defects appearing in the smoothing-cofactor and decoupled formulations.

\begin{proposition}\label{prop:homological_cofactor_identification}
Let $M(T(\mathscr T))$ be the global conformality matrix, let $M_1$ be the conformality matrix of the CNDC in the complete partition, and let $K$ be the condensed coupling matrix. Then
\begin{equation}\label{eq:homological_correction_identification}
\boxed{
 h_d(\mathscr T)
 =t(d+1)-\operatorname{rank}M(T(\mathscr T))
 =s(d+1)-\operatorname{rank}M_1
 =m-\operatorname{rank}K.}
\end{equation}
Consequently, Mourrain's homological correction term and the residual smoothing-cofactor correction term are not merely analogous: they are the same numerical invariant of the fixed T-mesh realization.
\end{proposition}

\begin{proof}
Mourrain's specialized formula and Lemma~\ref{lem:mourrain_topological_translation} give
\begin{equation}\label{eq:mourrain_specialized_dimension}
\dim S_d(\mathscr T)
=(d+1)^2+(c-t)(d+1)+n_v+h_d(\mathscr T).
\end{equation}
On the other hand, Theorem~\ref{thm: dim formula} gives
\[
\dim S_d(\mathscr T)
=(d+1)^2+c(d+1)+n_v-\operatorname{rank} M(T(\mathscr T)).
\]
Comparing the two  formulas proves the first equality in \eqref{eq:homological_correction_identification}.

The complete-partition formula in Theorem~\ref{thm:cndc_dim} gives
\[
\dim S_d(\mathscr T)
=(d+1)^2+(c+s-t)(d+1)+n_v-\operatorname{rank} M_1,
\]
and comparison with \eqref{eq:mourrain_specialized_dimension} proves
$h_d(\mathscr T)=s(d+1)-\operatorname{rank} M_1$.
Finally, Corollary~\ref{cor:cndc_spline_dim} gives
\[
\dim S_d(\mathscr T)
=(d+1)^2+(c-t)(d+1)+n_v+m-\operatorname{rank} K,
\]
which proves the last equality.
\end{proof}

We next translate Mourrain's ordered weights with the terminology in this paper. In his notation, $\operatorname{mis}(\mathscr T)$ is the set of maximal interior segments, namely the maximal segments that do not meet the boundary. Using the terminology of this paper, these segments are precisely the T $l$-edges. For an ordering $\iota$, Mourrain defines $\Gamma_{\iota}(l)$ as the vertices of $l$ that do not lie on a segment of larger order and defines the weight $\omega_{\iota}(l)$ as the sum of their degree--smoothness multiplicities. Since $d-(d-1)=1$, every such vertex has multiplicity one in the present setting, so $\omega_{\iota}(l)=|\Gamma_{\iota}(l)|$ where $|\Gamma_{\iota}(l)|$ denotes the number of vertices in $\Gamma_{\iota}(l)$.
\begin{lemma}\label{lem:mourrain_cndc_order_reduction}
Let
\[
D_{\iota}^{\mathrm M}(\mathscr T)
=\sum_{l\in T(\mathscr T)}[d+1-\omega_{\iota}(l)]_+
\]
be Mourrain's ordering correction term for all $t$ T $l$-edges. For a permutation $\pi$ of the $s$ CNDC T $l$-edges, let
\[
D_{\pi}^{\mathrm{CNDC}}
=\sum_{i=1}^{s}
[m(l_i)-n(l_i)+d+1-u_{\pi}^{-}(l_i)]_+.
\]
Then
\begin{equation}\label{eq:mourrain_cndc_min_equality}
\min_{\iota}D_{\iota}^{\mathrm M}(\mathscr T)
=
\min_{\pi}D_{\pi}^{\mathrm{CNDC}}.
\end{equation}
Moreover, after reversing one of the two ordering conventions when necessary, the local weights satisfy
\begin{equation}\label{eq:mourrain_weight_translation}
\omega_{\pi}(l_i)
=n(l_i)-m(l_i)+u_{\pi}^{-}(l_i).
\end{equation}
\end{lemma}

\begin{proof}
The quantity $n(l_i)-m(l_i)$ is the number of mono-vertices of $l_i$ relative to the CNDC. These include the endpoints, intersections with cross-cuts or rays, and intersections with T $l$-edges already removed into the diagonalizable component. Among the CNDC multi-vertices, exactly the $u_{\pi}^{-}(l_i)$ intersections with preceding CNDC edges belong to $\Gamma_{\pi}(l_i)$ under the corresponding ordering convention. Since each vertex has multiplicity one, this proves \eqref{eq:mourrain_weight_translation}, and hence
\[
[d+1-\omega_{\pi}(l_i)]_+
=[m(l_i)-n(l_i)+d+1-u_{\pi}^{-}(l_i)]_+.
\]

It remains to justify that the diagonalizable component contributes no additional optimized defect. The complete partition can be constructed by repeatedly removing a T $l$-edge that has at least $d+1$ vertices not shared with the T $l$-edges that remain. Record these removals in their elimination order. Given any CNDC permutation $\pi$, extend it by the reverse elimination order of the removed edges. Each removed edge then has at least $d+1$ vertices in its corresponding set $\Gamma$, so its defect is zero, while the restriction of the full correction to the CNDC is exactly $D_{\pi}^{\mathrm{CNDC}}$. This proves
\[
\min_{\iota}D_{\iota}^{\mathrm M}(\mathscr T)
\le \min_{\pi}D_{\pi}^{\mathrm{CNDC}}.
\]

Conversely, start with any full ordering. A removable diagonalizable edge may be moved to the appropriate end of the ordering without increasing the total defect: its at least $d+1$ private vertices force its own defect to be zero, and moving it out of the residual ordering can only make its intersection vertices available to the remaining edges, thereby increasing their weights and weakly decreasing their defects. Repeating this exchange for every removed edge produces a complete-partition-compatible ordering whose nonzero correction is the correction of the induced CNDC permutation. Therefore
\[
\min_{\pi}D_{\pi}^{\mathrm{CNDC}}
\le \min_{\iota}D_{\iota}^{\mathrm M}(\mathscr T),
\]
and \eqref{eq:mourrain_cndc_min_equality} follows.
\end{proof}

\begin{theorem}\label{thm:precise_mourrain_comparison}
Set
\[
B_0(\mathscr T)=(d+1)^2+(c-t)(d+1)+n_v.
\]
Mourrain's dimension formula and bounds are 
\begin{equation}\label{eq:mourrain_published_bounds_translated}
B_0(\mathscr T)
\le \dim S_d(\mathscr T)=B_0(\mathscr T)+h_d(\mathscr T)
\le B_0(\mathscr T)+\min_{\iota}D_{\iota}^{\mathrm M}(\mathscr T).
\end{equation}
The upper bound in \eqref{eq:mourrain_published_bounds_translated} coincides exactly with the upper bound in Theorem~\ref{thm:spline_dim_upper_bound}:
\begin{equation}\label{eq:upper_bound_exact_mourrain_equality}
U_K(\mathscr T)
=B_0(\mathscr T)+\min_{\pi}D_{\pi}^{\mathrm{CNDC}}
=B_0(\mathscr T)+\min_{\iota}D_{\iota}^{\mathrm M}.
\end{equation}
The lower bound $B_0(\mathscr T)$ is generally smaller than the lower bound in Proposition~\ref{prop:homological_cofactor_identification}. In fact 
\begin{equation}\label{eq:lower_bound_refined_mourrain_equality}
L_K(\mathscr T)
=B_0(\mathscr T)+[s(d+1)-n_{\textit{CNDC}}]_+.
\end{equation}
\end{theorem}

\begin{proof}
The upper bounds \eqref{eq:mourrain_published_bounds_translated} follow from Lemma~\ref{lem:mourrain_topological_translation} and the  inequalities $0\le h_d\le D_{\iota}^{\mathrm M}$~\cite{dim2014}. Equation \eqref{eq:upper_bound_exact_mourrain_equality} follows from Lemma~\ref{lem:mourrain_cndc_order_reduction}.

For the lower bound, Proposition~\ref{prop:homological_cofactor_identification} gives
\[
h_d(\mathscr T)=m-\operatorname{rank} K.
\]
The matrix $K$ has $m$ rows and
\[
q=n_{\textit{CNDC}}+m-s(d+1)
\]
columns. Therefore
\begin{align*}
h_d(\mathscr T)
&\ge m-\min\{m,q\}\\
&=[m-q]_+\\
&=[s(d+1)-n_{\textit{CNDC}}]_+,
\end{align*}
Adding the common topological term $B_0(\mathscr T)$ proves \eqref{eq:lower_bound_refined_mourrain_equality}.
\end{proof}

The preceding results explain both the equivalence and the additional structural information supplied by the present approach. Algebraically, the homological correction term is the nullity $m-\operatorname{rank} K$ of a coupling matrix obtained after all diagonalizable T $l$-edges have been eliminated. Combinatorially, the optimized correction term in the upper dimension bound can therefore be computed using only the $s$ CNDC edges, rather than all $t$ T $l$-edges. This is the precise sense in which the CNDC formulation is finer: it does not change the underlying correction invariant, but isolates it on the unique component where geometric rank variation can occur and exposes the extra matrix-size information used in the lower bound.

\section{Conclusion and future work}\label{sec:conclusion}

This paper develops a decoupling framework for the dimension of the spline spaces with the highest order of smoothness. The method separates the multi-vertex constraints in the CNDC and reduces the dimension problem to localized edge equations together with a small compatibility system.

Based on this framework, we obtain new dimension formulas and  explicit upper and lower bounds for the dimension of the spline space. These bounds are expressed through the refined structure that remains after the diagonalizable part of the T-mesh has been separated. The examples considered in this paper show that the bounds are sharp in the  sense that  different geometric realizations of the same T-mesh topology exist for which the spline-space dimension attains the lower and upper bounds, respectively.

We have further clarified the relationship between the present smoothing-cofactor approach and Mourrain's homological approach. Although the two methods are developed from different mathematical viewpoints, they lead to the same underlying decomposition of the dimension once the topology of the T-mesh is properly taken into account. The homological correction term is represented in the present framework by the remaining global compatibility conditions, while the refined CNDC description provides a more detailed interpretation of how these conditions arise from the incidence structure of the T $l$-edges. Thus, the present results do not merely restate the homological bounds, but provide a complementary and more localized explanation of their origin.

Several promising directions remain open for future research:
\begin{itemize}
    \item \textbf{Extension to Arbitrary and Mixed Order of Smoothness:} A natural problem is to extend the decoupling framework to polynomial splines with lower or mixed orders of smoothness ($\mu < d-1$), where the continuity constraints across adjacent cells exhibit more complex algebraic couplings.
    \item \textbf{Basis Function Construction:} Utilizing the localized linear equations derived from the decoupled cofactors, we plan to develop constructive algorithms for stable, locally supported basis functions over general non-diagonalizable T-meshes, which is essential for isogeometric analysis (IGA) applications.
    \item \textbf{Generalization to Higher Dimensions:} Extending the decoupling technique and CNDC analysis framework to three-dimensional T-meshes (trivariate splines) or unstructured meshes presents an important and challenging avenue for further exploration.
\end{itemize}

\section*{Acknowledgements}
This work is supported by the Key Project of National Natural Science Foundation of China (12494550,12494555), NSF of China(12371383). The authors declare that there are no conflicts of interest. 

\bibliographystyle{IEEEtran}
\bibliography{T1}

\end{document}